\documentclass[12pt,oneside,reqno,fleqn]{amsart}
\usepackage{amssymb}
\usepackage{bbm}
\usepackage{cases}
\usepackage{graphicx}
\usepackage{mathrsfs}
\usepackage{stmaryrd}
\usepackage{color}
\usepackage{soul}
\usepackage[dvipsnames]{xcolor}
\usepackage{amsfonts}
\usepackage{amsthm}
\usepackage{libertine}
\usepackage[T1]{fontenc}
\usepackage[libertine,varbb]{newtxmath}
\usepackage[colorlinks,linkcolor=Red,citecolor=Red,urlcolor=Red]{hyperref}
\usepackage[shortlabels] {enumitem}
\usepackage[nameinlink,capitalize]{cleveref}
\usepackage[spacing=true,kerning=true,tracking=true,letterspace=50]{microtype}
\usepackage{orcidlink}
\usepackage{ulem}
\usepackage[]{todonotes}
\hypersetup{pdftitle={Quantitative approximation of stochastic kinetic equations: from discrete to continuum}, pdfauthor={Zimo Hao, Khoa L\^e and Chengcheng Ling}}
\numberwithin{equation}{section}
\newcommand{\be}{\begin{eqnarray}}
\newcommand{\ee}{\end{eqnarray}}
\newcommand{\ce}{\begin{eqnarray*}}
\newcommand{\de}{\end{eqnarray*}}
\newtheorem{theorem}{Theorem}[section]
\newtheorem{lemma}[theorem]{Lemma}
\newtheorem{proposition}[theorem]{Proposition}
\newtheorem{corollary}[theorem]{Corollary}
\theoremstyle{remark}
\newtheorem{assumption}[theorem]{Assumption}

\newtheorem{example}[theorem]{Example}
\newtheorem{remark}[theorem]{Remark}
\newtheorem{definition}[theorem]{Definition}

\crefname{eqn}{Equation}{Equations}
\crefname{assumption}{Assumption}{Assumptions}
\crefname{innercustomthm}{Condition}{Conditions}
\crefrangelabelformat{innercustomthm}{#3#1#4-#5#2#6}

\def\bbk{{\boldsymbol{k}}}
\def\bbp{{\boldsymbol{p}}}
\def\bbr{{\boldsymbol{r}}}
\def\bbq{{\boldsymbol{q}}}
\def\bba{{\boldsymbol{a}}}

\def\bb2{{\boldsymbol{2}}}
\def\no{\nonumber}
\def\={&\!\!=\!\!&}

\def\e{{\mathrm{e}}}
\def\eps{\varepsilon}

\def\p{\partial}

\def\<{{\langle}}
\def\>{{\rangle}}
\def\({{\Big(}}
\def\){{\Big)}}

\def\bx{{\mathbf{x}}}

\def\sgn{\mbox{\rm sgn}}

\def\dif{d}

\def\min{{\mathord{{\rm min}}}}

\def\no{\nonumber}
\def\={&\!\!=\!\!&}
\def\bt{\begin{theorem}}
\def\et{\end{theorem}}
\def\bl{\begin{lemma}}
\def\el{\end{lemma}}
\def\br{\begin{remark}}
\def\er{\end{remark}}
\def\bd{\begin{definition}}
\def\ed{\end{definition}}
\def\bp{\begin{proposition}}
\def\ep{\end{proposition}}
\def\bc{\begin{corollary}}
\def\ec{\end{corollary}}
\def\bx{\begin{example}}
\def\ex{\end{example}}

\def\cC{{\mathcal C}}
\def\cD{{\mathcal D}}

\def\cR{{\mathcal R}}
\def\cS{{\mathcal S}}

\def\mB{{\mathbb B}}

\def\mE{{\mathbb E}}
\def\E{\mE}

\def\mL{{\mathbb L}}

\def\mN{{\mathbb N}}

\def\mP{{\mathbb P}}
\def\mQ{{\mathbb Q}}
\def\mR{{\mathbb R}}

\def\sF{{\mathscr F}}

\def\sI{{\mathscr I}}

\def\bC{{\mathbb C}}
\def\bB{{\mathbb B}}

\def\geq{\geqslant}
\def\leq{\leqslant}

\def\DD{{b}}

\newcommand{\R}{{\mathbb R}}
\newcommand{\smooth}{C^\infty_c}
\newcommand{\Rd}{{\R^d}}

\newcommand{\tand}{\quad\text{and}\quad}

\newcommand{\tif}{\quad\text{if}\quad}
\newcommand{\twhere}{\quad\text{where}\quad}
\newcommand{\LL}{{\mathbb{L}}}

\newcommand{\1}{{\mathbf 1}}

\newcommand{\norm}[1]{{\left\vert\kern-0.25ex\left\vert\kern-0.25ex\left\vert #1 
    \right\vert\kern-0.25ex\right\vert\kern-0.25ex\right\vert}}
\DeclareMathOperator*{\esssup}{ess\,sup}
\renewcommand{\le}{\leq}
\renewcommand{\ge}{\geq}
\allowdisplaybreaks
\begin{document}
	\title{Weak approximation of kinetic SDEs: closing the criticality gap}
	\date{\today}
	
\author{Zimo Hao, Khoa L{\^e} 
\and Chengcheng Ling 
}
 
\address{Universit\"at Bielefeld, Fakult\"at f\"ur Mathematik, 33615 Bielefeld, Germany}
\email{zhao@math.uni-bielefeld.de}
\address{
University of Leeds, School of Maths, 
Leeds, LS2 9JT, UK.}
\email{ K.Le@leeds.ac.uk}
\address{ University of Augsburg, Institut f\"ur Mathematik, 
86159 Augsburg,  Germany.}
\email{ chengcheng.ling@uni-a.de}

	\begin{abstract}
We study the weak convergence of a generic tamed Euler-Maruyama  scheme for the kinetic stochastic differential equations (SDEs) 
with integrable drifts. We show that the marginal density of the considered scheme converges at rate $1/2$ to the corresponding marginal density of the SDE. The convergence rate is independent from the criticality gap, which is new compared to previous results.
		
		\bigskip
		

	\end{abstract}
	
	\maketitle
\section{Introduction}
Let $d\ge1$ be an integer dimension, $\xi,\eta$ be vectors in $\Rd$ and $W$ be a $d$-dimensional standard Brownian motion on a filtered probability space  $(\Omega,\sF,(\sF_s)_{s\ge0},\mP)$.
We consider the second order stochastic differential system 
\begin{align}\label{eq:SDE}
\begin{cases}
\dif X_t=V_t\dif t,\\
\dif V_t=b(t,X_t,V_t)\dif t+\dif W_t,
\end{cases}
\quad (X_0,V_0)=(\xi,\eta)\in\Rd\times\Rd.
\end{align}
The drift $b :\mR_+\times\mR^{d}\times\mR^{d}\to\mR^d$ is a  measurable function which satisfies the following integrability condition. 
\begin{assumption}
\label{ass:main}There exists a constant $\bbp=(p_x,p_v)\in[2,\infty]^2$ such that
\begin{align}
   \bba\cdot\frac{d}{\bbp}:= 3\frac{d}{p_x}+\frac{d}{p_v}<1
\tand
   \|b\|_{L_T^\infty(\mL^{\bbp})}:= \sup_t \left(\int_{\R^{d}}\left(\int_{\Rd} |b(t,x,v)|^{p_x}dx\right)^{\frac{p_v}{p_x}}dv\right)^{\frac1{p_v}}<\infty. 
\end{align}
(The Lebesgue integration norm is replaced by $\sup$-norm when $p_x=\infty$ or $p_v=\infty$.)
\end{assumption}
System \eqref{eq:SDE} is one of the typical models that describes the Hamiltonian mechanics in the form of Langevin equation (\cite{S,T}). In such context,  $X_t,V_t$ usually represent respectively the position and velocity of a moving particle at a time $t$. 
The probability density $\rho_t(x,v)$  of $Z_t:=(X_t,V_t)$ (which exists under \cref{ass:main}, see \cite{RZ24}) is a common subject in statistical mechanics, and it satisfies the following Fokker–Planck equation (\cite{RE, V}) 
 \begin{align}\label{eq:PDE-intro}
    \partial_t\rho=\Delta_v\rho +v\cdot\nabla_x \rho+b\cdot\nabla_v\rho.
 \end{align}
 

Let $n\ge1$ be an integer. 
 We define $k_n(t):=\frac{\lfloor nt \rfloor}{n}$ for each $ t\in \mR_+$. 
 We  let $(b_n)_n$ be a sequence of bounded functions that converges to $b$.
Let $\{Z^n_t\}:=\{(X^n_t,V^n_t)\}$ be the solution to the following \textit{tamed Euler scheme}
\begin{align}\label{eq:SDE-EM}
\begin{cases}
X^n_t=\xi+\int_0^t V^n_s\dif s,\\
V^n_t=\eta+\int_0^t b_n(s,X^n_{k_n(s)}+(s-k_n(s))V^n_{k_n(s)},V^n_{k_n(s)})\dif s+W_t.
\end{cases}
\end{align}
The probability density of $Z^n_t$, which exists, is denoted by $\rho^n_t$. The goal of the current article is to obtain convergence of $\rho^n_t$ to $\rho_t$ with a rate that is independent from $d,p_x,p_v$. 
To be more precise, we assume that $(b_n)_n$ satisfies the following condition.
\begin{assumption}\label{ass.bn}
    There are finite constants   $\zeta\in[0,\frac12]$, $\vartheta>0$, $\delta\in(1,2)$ and $\kappa_b>0$ such that 
  \begin{align}\label{con:bnbdd}
&\sup_n\|b_n\|_{L_T^\infty(\mL^{\bbp})}\le \kappa_b,
\\& \sup_n n^{\vartheta\delta} \|b_n-b\|_{L^\infty((0,1],\mB^{-\delta,1}_{\bbp;\bba})}\le \kappa_b ,\label{con.bnrate}
\\&
    \sup_{n} n^{-\zeta}\sup_{(t,z)\in[0,n^{-1}]\times\R^{2d}}|b_n(t,z)|\le \kappa_b,\label{con.taming}
\end{align}
and 
\begin{align}
    \lim_{h\downarrow 0} \sup_n(n^{-1/2}\wedge h)\sup_{(t,z)\in[0,1]\times\R^{2d}}|b_n(t,z)|=0.\label{con.bngir}
\end{align}
\end{assumption}
In \eqref{con.bnrate}, $\mB^{-\delta,1}_{\bbp;\bba}$ is an anisotropic Besov space of distributions on $\R^{2d}$ with regularity index $-\delta$.
Precise definitions are provided later in \cref{sec:Notation-result}. Heuristically,  condition \eqref{con:bnbdd} asserts that $b_n$ belongs to the same function space as $b$ uniformly. The constant $\vartheta$ parametrizes and quantifies the approximating sequence $(b_n)_n$. Condition \eqref{con.bngir} is a technical one which is necessary for the application of Girsanov theorem.
\bt\label{thm-weak} Assume that \cref{ass:main,ass.bn} hold and initial data $(\xi,\eta)\in\Rd\times\Rd$ is given. 
Then for
  any $\bbq\in[2,\infty]^2$ with $\bbq\ge\bbp$, 
  there is a constant $C=C(d,\bbp,\bbq,\zeta,\vartheta,\delta,\kappa_b,\|b\|_{L_T^\infty(\LL^\bbp)})>0$ such that for all $t\in(0,1]$, 
     \begin{align}
    \label{est:thm-weak-S-1}
 \|\rho_t-\rho_t^n\|_{\bbq'}
 \lesssim_C n^{-(\vartheta\wedge\frac12)}t^{-(\zeta\vee\bba\cdot \frac{d}{2\bbp})-\bba\cdot\frac{d}{2\bbp}},
\end{align}
where $\bbq'=(\frac{q_x}{q_x-1},\frac{q_v}{q_v-1}) $ is the H\"older conjugate of $\bbq.$
\et
We note that the constant $C$ in \eqref{est:thm-weak-S-1} is independent from the initial position. 
The restriction on the time interval $[0,1]$ is conventional which simplifies our exposition. The results discussed herein can be extended straightforwardly to arbitrary time intervals. In such case, the implicit constant in \cref{thm-weak} would depend on the the length of the temporal interval.
We give two typical examples of $(b_n)_n$ which satisfies \cref{ass.bn}. 
The first one is given by the convolution
\begin{equation}\label{def.bnconvl}
    b_n=b*\phi_n, \quad \phi_n(x,v):=n^{4d\vartheta}\phi(n^{3\vartheta}x, n^{\vartheta}v),
\end{equation}
where $\phi$ is any  bounded  probability density function  with $\varphi(x,v)=\varphi(x,-v)$  and $\vartheta$ is any constant in $(0,\frac12(\bba\cdot \frac{d}{\bbp})^{-1}]$. 
When $((p_x\wedge p_v)-1)\bba\cdot\frac{d}{\bbp}>1$ or when $\bbp=(\infty,\infty)$,  the other example of taming sequence is obtained through the truncation 
\begin{align}
           b_n(t,x,v)&:=
            1_{|b(t,x,v)|>0}\frac{|b(t,x,v)|\wedge(C_2n^{\kappa})}{|b(t,x,v)|}b(t,x,v),\label{def.bncutoff}
        \end{align}
where  $\kappa$ is any constant in $(0,1/2)$. 
To obtain rate $1/2$ in \eqref{est:thm-weak-S-1}, it is sufficient to tune the parameters so that  $\vartheta>\frac{1}{2}$ in case of \eqref{def.bnconvl} and $\kappa>\frac{1}{2}(\bba\cdot\frac{d}{\bbp})$ in case of \eqref{def.bncutoff}, see details in \cref{sec.examples} below.

\subsection*{Discussion}
Comparing \eqref{eq:SDE-EM} with the standard Euler-Maruyama scheme 
\begin{align}\label{eq.standardEM}
\begin{cases}
X^n_t=\xi+\int_0^t V^n_{k_n(s)}\dif s,\\
V^n_t=\eta+\int_0^t b(s,Z^n_{k_n(s)})\dif s+W_t,
\end{cases}
\end{align}
we have replaced   $\int_0^tV_{k_n(s)}^nds$ and $ \int^t_0 b(s,Z^n_{k_n(s)})\dif s $ respectively by $\int_0^tV_s^nds$ and $ \int_0^t b_n(s,X^n_{k_n(s)}+(s-k_n(s)V^n_{k_n(s)},V^n_{k_n(s)})\dif s$. 
This replacement is more favorable due to the several reasons. 

First, even though $\int_0^tV^n_sds$ has continuum variable, the scheme \eqref{eq:SDE-EM} is still recursive. Indeed, let $h:=1/n$, for any $k\in\mN$, we obtain from \eqref{eq:SDE-EM} that
\begin{align*}
     X_{(k+1)h}^n&=X_{kh}^n+\int_{kh}^{(k+1)h}V_{s}^n\dif s,
     \\
     V_s^n&=V_{kh}^n+\int_{kh}^s b^n(r,X_{kh}^n+(r-kh)V_{kh}^n,V_{kh}^n)\dif r +W_s-W_{kh},\ s\in[kh,(k+1)h].
\end{align*}
This implies the recursive relation
\begin{align}\label{eq.recursive}
\left(\begin{matrix} X_{(k+1)h}^n \\ V_{(k+1)h}^n \end{matrix}\right)=
F^n_h(X_{kh}^n,V_{kh}^n)+\xi^n_k,
\end{align}
where
\begin{align*}
    F^n_h(x,v):= \left(\begin{matrix} x+hv+\int_{0}^{h}(h-r) b^n(r,x+rv,v)\dif r \\ v+\int_{0}^{h}b^n(r,x+rv,v)\dif r \end{matrix}\right)
\end{align*}
and $\xi^n_k=\left(\begin{matrix} \int_{kh}^{(k+1)h}(W_s-W_{kh})\dif s\ \ W_{(k+1)h}-W_{kh} \end{matrix}\right)$.
We note that $\xi^n_k$ are i.i.d. Gaussian random variables with covariance matrix 
\begin{align*}
   \left(\begin{matrix}\frac{h^3}{3}\mathbb{I}_{d\times d} && \frac{h^2}{2}\mathbb{I}_{d\times d}
   \\ \frac{h^2}{2}\mathbb{I}_{d\times d} && h\mathbb{I}_{d\times d}
    \end{matrix}\right).
\end{align*}
Consequently, in the case of time-independent smooth drift $b$, the scheme \eqref{eq:SDE-EM} is completely implementable through \eqref{eq.recursive}. 

    Second,   we use tamed term $b_n$ instead of $b$ itself,  which is known as ``taming technique'' \cite{HJK, LL}, to aid the hardly controllable  case when $b$ has singularities. The taming of the drift turns out to be technically beneficial in analyzing convergence rates.

\smallskip
There are limited works on the study of the weak well-posedness of  \eqref{eq:SDE}. Sharp conditions given from  \cite[Theorem 1]{RM}  are either $b\in L^\infty_T(C_{x,v}^\beta)$ with $ \beta\in(0,1)$ or $b\in L^\infty_T(\mL^\infty)$ or $L_T^q(\mL^p)$ with $\frac{2}{q}+\frac{4d}{p}<1$ and $p\geq 2,q>2$.   \cite[Theorem 1.4]{RZ24} obtains, among other things, weak well-posedness for \eqref{eq:SDE} with unbounded singular drifts in Kato’s class,
 which includes  \cref{ass:main}.  
Considering the fact that there are few results on weak well-posedness, it seems that the current result  is the first work that quantifies the weak convergence rate of order $1/2$ for \eqref{eq:SDE} under \cref{ass:main}. The work \cite{LM} also studies the law of the discrete Euler--Maruyama equation but does not provide any quantitative bounds.  
        
When $b=b(t,v)\in \mL^{(p_x,p_v)}$ with $p_x=\infty$ and $p_v>d$, the kinetic SDE \eqref{eq:SDE} reduces to the  non-degenerate SDE $\dif V_t=b(t,V_t)\dif t+\dif W_t$. In this particular instance, we can compare \cref{thm-weak} with existing literature. 
In fact, the work \cite{JM21} considered an Euler scheme with randomized time variable in the approximating drift
 \begin{align}\label{11111}
     {b}_n(t,v):=1_{t>1/n,  |b(t,v)|>0}\frac{|b(t,v)|\wedge(Cn^{\kappa})}{|b(t,v)|}b(t,v), \quad \kappa=\frac d{p_v},
\end{align}
and established a pointwise convergence rate for $\rho^n$ of order $n^{-\frac{1 - d/p_v}{2}}$, which vanishes when $d/p_v$ approaches $1$. 
Due to technical reasons, \cite{JM21} requires that $b_n$ vanishes when $t<1/n$, which causes  \eqref{11111} to violate \eqref{con.bnrate}. 
In contrast, our method does not treat the initial time step differently and therefore allows the use of \eqref{def.bncutoff}. In this case, by  choosing $\kappa >\frac{1}{2}(\bba\cdot\frac{d}{\bbp})$, \cref{thm-weak} yields a convergence rate (with respect to a weaker topology) of order $n^{-\frac12}$,
which is independent from the criticality gap $1- d/{p_v}$. 
In the case when $p_v=\infty$, the criticality gap  incidentally vanishes, our result yields the convergence rate in the total variation distance, which is consistent with \cite{OJ}.
\cite{JM21} also considers the case  \eqref{11111} with $\kappa=1/2$, the choice \eqref{def.bncutoff} with $\kappa=1/2$ is ruled out by condition \eqref{con.bngir} (see \cref{rmk.girpara} for further details).  

\subsection*{Idea of the proof.} We briefly sketch the general arguments showing \eqref{est:thm-weak-S-1}. Via a duality identity (\cref{lem:app}), it suffices to estimate $|\mE \varphi(Z_t)-\mE \varphi(Z_t^n)|$, for each test function $\varphi$ satisfying $\|\varphi\|_{\mL^\bbq}=1$. Let $(P_t)_t$ be the free Markovian semigroup associated to \eqref{eq:SDE}.   Applying It\^o's formula to $r\to P_{t-r}\varphi(Z_r)$ and $P_{t-r}\varphi(Z^n_r)$, we have 
\begin{align*}
    |\mE \varphi(Z_t)-\mE \varphi(Z_t^n)|\le& \Big|\mE\int_0^t (b-b_n)(r,Z_r)\cdot \nabla_vP_{t-r}\varphi(Z_r)\dif r\Big|\no\\
    &+\Big|\mE\int_0^t b_n(r)\cdot \nabla_vP_{t-r}\varphi(Z_r)-b_n(r)\cdot \nabla_vP_{t-r}\varphi(Z_r^n)\dif r\Big|\no\\
    &+\Big|\mE\int_0^t \big(b_n(r,Z^n_r)-\Gamma_{r-k_n(r)}b_n(r,Z_{k_n(r)}^n)\big)\cdot \nabla_vP_{t-r}\varphi(Z_r^n)\dif r\Big|\no\\
    =:&I_1^n+I_2^n+I_3^n.
\end{align*}
The first term, $I_1^n$, encodes the taming error arising from approximating $b$ by $b_n$. The convergence rate of $I^1_n$ is therefore governed by \eqref{con.bnrate}. To utilize \eqref{con.bnrate}, one must effectively control the regularity of the product $ \nabla_vP_{t-s}\varphi\cdot (b-b_n)(s)$. Standard product estimates would yield unsatisfactory rates. We overcome this difficulty by using Bony's paraproduct decomposition, writing $\nabla_v P_{t-s}\varphi\prec(b_n-b)(s)+\nabla_v P_{t-s}\varphi\succcurlyeq(b_n-b)(s)$. Each term in Bony's decomposition possesses a different degree of regularity and integrability, and by treating them separately, we are able to apply  condition \eqref{con.bnrate} effectively (see \cref{lem:estIn2}). 
The second term, $I_2^n$, has a recursive structure and can be estimated directly in terms of the map $s\mapsto \|\rho_s-\rho^n_s\|_{\bbp'}$. By applying a Gr\"onwall argument, this recursive contribution becomes negligible. The last term, $I^n_3$, is a weighted quadrature error arising from the discretizations of the time steps. To control it, we use refined estimates related to the transition probabilities associated with $Z$ and $Z^n$, see \cref{lem.technical,lem:I3,lem:est-I3n}.  

The application of It\^o's formula leading to the decomposition into $I^n_1,I^n_2,I^n_3$ is already presented in literature. For instance,  it is used in \cite{Holland} in a slightly different context, where the test function $\varphi$ is H\"older continuous, which leads to convergence in a Wasserstein metric. It is also similar to the classical approach in \cite{talay1990expansion}.
Our analysis of each term is nevertheless novel, and new techniques have been introduced to handle the lack of regularity present in the coefficients and the test functions.

\subsection*{Organization of the paper}
Precise definitions, notations and examples are stated and discussed in \cref{sec:Notation-result}. \cref{sec.pre} collects some analytic estimates on the free Markov semigroup and exponential functionals. The proof of \cref{thm-weak} is provided in \cref{sec.proofs}.  \cref{app} collects several technical lemmas for calculations appeared in the main proofs. Paraproduct estimates in anisotropic Besov spaces are summarized in \cref{sec:3.2}.  
\subsection*{Conventions}
A vector $z$ in $\R^{2d}$ is often decomposed into two components $z=(x,v)\in\Rd\times\Rd$. We denote $\nabla_x=(\partial_{x_k})_{1\le k\le d}$, $\nabla_v=(\partial_{v_k})_{1\le k\le d}$ and $\nabla=(\nabla_x,\nabla_v)$. 



We use relations $<,=, \leq$ and  $>,\geq$ between vectors if each corresponding component shares the same relation, e.g. 
 $\bbq:=(q_i)_{i\geq 1}\leq \bbp:=(p_i)_{i\geq 1}$ if $q_i\leq p_i$ for each $i$.
For each $\bbp=(p_x,p_v)\in[1,\infty]^2$, we denote $\bbp^{-1}:=\frac{1}{\bbp}:=(\frac{1}{p_x},\frac{1}{p_v})$, with convention that $1/\infty=0$. 

The notation $a\lesssim b$ means that $a\leq C b$ for some finite non-negative constant $C$ which depends only on the parameters in the corresponding statement. If there are further dependence on another parameter $c$, we incorporate it in the notation by writing $a\lesssim_c b$.

\section{Notations and examples}\label{sec:Notation-result}
\subsection{Notations}
For each $p\in[1,\infty)$,  $\|\cdot \|_p$ denotes the classical Lebesgue norm on $\Rd$, while $\|\cdot\|_\infty$ denotes the supremum norm over $\Rd$, i.e. $\|f\|_\infty:=\sup_{x\in\mR^d}|f(x)|$. We note that this differs from the  usual $L^\infty$ Lebesgue norm defined via the essential supremum. This choice is purely for convenience since we  consider only  functions that are bounded everywhere.  

 
 For each vector index $\bbp=(p_x,p_v)\in[1,\infty]^2$, we define the mix $\mL^\bbp$ space as the collection of measurable functions $f$ on $\R^{2d}$ such that
 \begin{align*}
\|f\|_{\bbp}:=\left(\int_{\mR^d}\|f(\cdot,v)\|_{p_x}^{p_v}\dif v\right)^{1/p_v}<\infty \quad\text{if}\quad p_v<\infty,
\end{align*}
and 
\begin{align*}
\|f\|_{\bbp}:=\sup_{v\in\mR^d}\|f(\cdot,v)\|_{p_x}<\infty \quad\text{if} \quad p_v=\infty.
\end{align*}
When $\bbp=(\infty,\infty)$, we abbreviate $\|\cdot\|_{\bbp}=\|\cdot\|_\infty$.


For each $q\in[1,\infty]$ and $\bbp=(p_x,p_v)\in[1,\infty]^2$, $L^q_T(\mL^\bbp)$ denotes the space of measurable functions $f$ on $[0,1]\times\R^{2d}$ such that
\begin{align*}
    \|f\|_{L_T^q(\mL^\bbp)}:&=\Big(\int_0^T\|f(t,\cdot)\|_{\bbp}^{q}\dif t\Big)^{1/q}<\infty \tif q<\infty \quad \text{and}
    \\
    \|f\|_{L_T^\infty(\mL^\bbp)}:&=\sup_{t\in[0,1]}\|f(t,\cdot)\|_{\bbp}^{q}<\infty \tif q=\infty.
\end{align*}
\subsection*{Anisotropic Besov spaces} \label{sec.AS}
For a Lebesgue integrable function $f$ in $\mR^{2d}$, let $\hat f$ and $\check f$ respectively be the Fourier transform of $f$ and its inverse which are defined by
$$
\hat f(\xi):=(2\pi)^{-2d}\int_{\mR^{2d}} \e^{-{\rm i}\xi\cdot z}f(z)\dif z,
\quad \check f(\xi):=(2\pi)^{-2d}\int_{\mR^{2d}} \e^{{\rm i}\xi\cdot z}f(z)\dif z,\quad\xi\in\mR^{2d}.
$$
Put $\bba=(3,1)$ and define the anisotropic distance 
$$
|z-z'|_\bba:=|x- x'|^{1/3}+|v-v'|, \quad \forall z=(x,v),z'=(x',v')\in\mR^{d}\times\Rd.
$$
For each $r>0$ and $z\in\mR^{2d}$,  we denote
$
B^\bba_r(z):=\{z'\in\mR^{2d}:|z'-z|_\bba\leq r\}
$
and $ B^\bba_r:=B^\bba_r(0)$.
Let $\chi^\bba_0$ be  a 
$C^{\infty}_c$-function  on $\mR^{2d}$ which is symmetric in the direction of $x$ and $v$ with $\chi^\bba_0(\xi)=1$ if $\xi\in B^\bba_1$ and $\chi^\bba_0(\xi)=0$ if $\xi\notin B^\bba_2$.
We define
$$
\phi^\bba_j(\xi):=
\left\{
\begin{aligned}
&\chi^\bba_0(2^{-j\bba}\xi)-\chi^\bba_0(2^{-(j-1)\bba}\xi),\ \ &j\geq 1,\\
&\chi^\bba_0(\xi),\ \ &j=0,
\end{aligned}
\right.
$$
where for each $s\in\mR$ and $\xi=(\xi_1,\xi_2)\in\Rd\times\Rd$,
$
2^{s\bba }\xi=(2^{3s}\xi_1, 2^{s}\xi_2).
$
Note that
\begin{align*}
{\rm supp}(\phi^\bba_j)\subset\big\{\xi: 2^{j-1}\leq|\xi|_\bba\leq 2^{j+1}\big\},\ j\geq 1,\quad {\rm supp}(\phi^\bba_0)\subset B^\bba_2,
\end{align*}%
and
\begin{align}\label{AA13}
\sum_{j=0}^k\phi^\bba_j(\xi)=\chi_0^\bba(2^{-k\bba}\xi)\to 1,\, \text{as $k\to \infty$},\quad \forall\xi\in\mR^{2d}.
\end{align}

Let $\cS$ be the space of all Schwartz functions on $\mR^{2d}$ and $\cS'$ be its topological dual, i.e. the space of  tempered distributions.
For given $j\geq 0$, the  dyadic anisotropic block operator  $\mathcal{R}^\bba_j$ is defined on $\cS'$ by
\begin{align*}
\mathcal{R}^\bba_jf(z):=(\phi^\bba_j\hat{f})\check{\ }(z)=\check{\phi}^\bba_j*f(z),
\end{align*}%
where the convolution is understood in the distributional sense.
We conventionally put $\cR^\bba_i:=0$ for $i<0$.
We give a definition for anisotropic Besov spaces (cf. \cite[Chapter 5]{Tri06}).
\begin{definition}\label{bs}
Let $s\in\mR$, $q\in[1,\infty]$ and $\bbp=(p_x,p_v)\in[1,\infty]^2$. The  anisotropic Besov space is defined by 
$$
\bB^{s,q}_{\bbp;\bba}:=\left\{f\in \cS': \|f\|_{\bB^{s,q}_{\bbp;\bba}}
:= \left(\sum_{j\geq0}\big(2^{ js}\|\cR^\bba_{j}f\|_{\bbp}\big)^q\right)^{1/q}<\infty\right\}.
$$
When $q=\infty$, we write $\bB^s_{\bbp;\bba}:=\bB^{s,\infty}_{\bbp;\bba}$.
\end{definition}

For any $\bbp\in[1,\infty]^2$,  $s'>s$ and $q\in[1,\infty]$, it holds that (\cite[Appendix B]{HRZ})
\begin{align}\label{AB2}
\bB^{0,1}_{\bbp;\bba}\hookrightarrow\mL^\bbp\hookrightarrow\bB^{0,\infty}_{\bbp;\bba},\quad  \bB^{s',\infty}_{\bbp;\bba}\hookrightarrow \bB^{s,1}_{\bbp;\bba}\hookrightarrow \bB^{s,q}_{\bbp;\bba}.
\end{align}


\begin{lemma}\label{lem.nablabes}
For any $\bbp\in [1,\infty]^2$ and $m,n\in\mN_0$, there is a constant $C=C(\bbp,d)>0$ such that for any $f\in \bB^{m+3n,1}_{\bbp;\bba}$,
    \begin{align}\label{est:nabla}
    \|\nabla^m_v \nabla^n_x f\|_{\bbp}\lesssim \|f\|_{\bB^{m+3n,1}_{\bbp;\bba}}.
\end{align}
\end{lemma}
\begin{proof}
    We will make use of Bernstein's inequalities (see \cite{ZZ21}): for any $j\ge0$,
    \begin{align}\label{Ber}
\|\nabla^{k_1}_x\nabla^{k_2}_v\cR^\bba_j f\|_{\bbp'}\lesssim 2^{j \bba\cdot(\bbk+\frac{d}{\bbp}-\frac{d}{\bbp'})}\|\cR^\bba_j  f\|_{\bbp},\qquad  \|\cR^\bba_jf   \|_{\bbp}\lesssim \|f\|_{\bbp}.
\end{align}
    It follows that
    \begin{align*}
        \|\nabla^m_v \nabla^n_x f\|_{\bbp}\le \sum_{j\ge-1}\|\nabla^m_v \nabla^n_x \cR_j^\bba f\|_{\bbp}\lesssim  \sum_{j\ge-1}2^{j(m+3n)}\|\nabla^m_v \nabla^n_x \cR_j^\bba f\|_{\bbp}\lesssim \|f\|_{\bB^{m+3n,1}_{\bbp;\bba}},
    \end{align*}
    as desired.
\end{proof}

\subsection{Examples}\label{sec.examples}
We discuss two specific examples of \eqref{eq:SDE-EM} and their  corresponding convergence rates. 
\begin{example}[Convolution]
     We show that $ b_n$ defined in \eqref{def.bnconvl} satisfies \cref{ass.bn} for any $\vartheta\in(0,\frac12(\bba\cdot \frac{d}{\bbp})^{-1}]$. We see that \eqref{con:bnbdd} holds because
  \begin{align*}
      \sup_n\|b_n\|_{L^\infty_T(\mL^{\bbp})}\le\|b\|_{L^\infty_T(\mL^{\bbp})} <\infty.
  \end{align*}
  For each $N>1$, Young's convolution inequality implies that
  \begin{align*}
    \|b_n\|_{L_T^\infty(\mL^\infty)}&\lesssim \|(b1_{|b|<N})*\phi_n\|_{L_T^\infty(\mL^\infty)}+\|(b1_{|b|\ge N})*\phi_n\|_{L_T^\infty(\mL^\infty)}\\
    &\lesssim \|b1_{|b|<N}\|_{L_T^\infty(\mL^\infty)}+n^{\vartheta\bba\cdot\frac{d}{\bbp}}\|b1_{|b|\ge N}\|_{L_T^\infty(\mL^{\bbp})}\le N+n^{\vartheta\bba\cdot\frac{d}{\bbp}}\|b1_{|b|\ge N}\|_{L_T^\infty(\mL^{\bbp})}.
  \end{align*}
  Hence,
  \begin{align}\label{test:bn}
      \lim_{h\downarrow 0} \sup_n(n^{-1/2}\wedge h)\|b_n\|_{L_T^\infty(\mL^\infty)}&\lesssim \lim_{h\downarrow 0}hN^2+ n^{-1/2}n^{\vartheta\bba\cdot\frac{d}{\bbp}}\|b1_{|b|\ge N}\|_{L_T^\infty(\mL^{\bbp})}\nonumber\\&\le \|b1_{|b|\ge N}\|_{L_T^\infty(\mL^{\bbp})},
  \end{align}
  which vanishes 
  as $N\to \infty$. This shows \eqref{con.bngir}. A similar application of Young's convolution inequality also shows that \eqref{con.taming} with $\zeta=\vartheta\bba\cdot\frac{d}{\bbp}$. Moreover, \cref{lem:f-fn} implies that for any $\delta\in(1,2)$,
  \begin{align*}
       \|b_n-b\|_{L^\infty([0,1],\mB^{-\delta,1}_{\bbp;\bba})}\le c_\delta n^{-\delta\vartheta},
  \end{align*}
  which shows \eqref{con.bnrate}.

Thus, by choosing $\vartheta\in[\frac12,\frac12(\bba\cdot\frac{d}{\bbp})^{-1}]$, \eqref{est:thm-weak-S-1} yields
\begin{align*}
       \|\rho_t-\rho^n_t\|_{\bbq'}\lesssim_C \left(t^{-(\vartheta+\frac12)\bba\cdot\frac{d}{\bbq}}+t^{-\frac12(\bba\cdot\frac{d}{\bbp}+\bba\cdot\frac{d}{\bbq})}\right)n^{-\frac12}.
\end{align*}
\end{example}
\begin{example}[Cut-off]\label{ex.cutoff}
    We show that  $b_n$ defined by \eqref{def.bncutoff} satisfies \cref{ass.bn} with $\zeta=\kappa$ and $\vartheta=\kappa(\bba\cdot\frac{d}{\bbp})^{-1}$.
    Obviously,
     \begin{align*}
      \sup_n\|b_n\|_{\mL^\infty_T(\mL^{\bbp})}\le\|b\|_{\mL^\infty_T(\mL^{\bbp})}<\infty,
  \end{align*}
  so that \eqref{con:bnbdd} holds.
    When $\bbp\neq (\infty,\infty)$, applying \cref{App:cutoff}, we have
\begin{align*}
\|b_n-b\|_{L^\infty([0,1],\mB^{-\delta,1}_{\bbp;\bba})}\lesssim \|b1_{|b|>n^{\kappa}}\|_{L^\infty([0,1],\mB^{-\delta}_{\bbp;\bba})}\lesssim n^{-\delta\kappa(\bba\cdot d/\bbp)^{-1}},
\end{align*}
provided that $\delta\in(0,((p_x\wedge p_v)-1)\bba\cdot\frac{d}{\bbp})$. 
To see that \eqref{con.bngir} holds, we note that $\|b_n\|_{\infty}\lesssim n^\kappa$ so that
   \begin{align*}
      \lim_{h\downarrow 0} \sup_n(n^{-1/2}\wedge h)\|b_n\|_{\infty}\le C\lim_{h\downarrow 0} \sup_n(n^{-1/2}\wedge h)n^{\kappa}=0
  \end{align*}
  provided that $\kappa<\frac12$.  Lastly, \eqref{con.taming} holds with $\zeta=\kappa$ because
  \begin{align*}
   \sup_{n\in\mR^d} n^{-\kappa}\sup_{t\in[0,n^{-1}]}\|b_n(t)\|_\infty\lesssim 1.
  \end{align*}
  Hence in this case, we obtain from \cref{thm-weak} that 
  \begin{align*}
       \|\rho_t-\rho^n_t\|_{\bbq'}\lesssim_C 
       \left(t^{-\kappa-\frac12\bba\cdot\frac{d}{\bbq}}+t^{-\frac12(\bba\cdot\frac{d}{\bbp}+\bba\cdot\frac{d}{\bbq})}\right)(n^{-\frac12}+n^{-\kappa(\bba\cdot\frac{d}{\bbp})^{-1}}).
  \end{align*}
\end{example}

\section{Preliminaries}\label{sec.pre}
\subsection{Estimates on Markov semigroup}\label{sec:Tools}
For each $t>0$, $z=(x,v)\in\Rd\times\Rd$ and any bounded measurable function $f:\R^{2d}\to\R$, we denote
\begin{align}\label{def.GM}
    &G_t:=\Big(\int_0^t W_sds, W_t\Big), \quad  M_t(z):=G_t+\Gamma_tz,
    \\ \label{def:pt-intro}
        &P_tf(z):=\mE f\left(x+tv+\int_0^t W_s\dif s,v+W_t\right)=\mE f\big(M_t(z)\big).
\end{align} 
The process $(M_t(z))_t$ is the solution starting at $z$ to SDE \eqref{eq:SDE} without the drift term (i.e. the underlying free process) and $(P_t)_t$ is its associated Markov semigroup.

   Let $g_t$ be the probability density of $G_t$, which is given by 
\begin{align}
\label{def.densityG}g_t(x,v):=\left(\frac{\pi^2 t^4}{3}\right)^{-d/2}\exp\left(-\frac{3|x|^2+|3x-2tv|^2}{2t^3}\right).
\end{align}
It is straightforward to obtain the following relations
\begin{align}\label{828:00}
&g_t(x,v)=t^{-2d}g_1(t^{-\frac{3}{2}}x,t^{-\frac{1}{2}}v),\nonumber\quad \\ &\Gamma_tg_t(x,v)=t^{-2d}g_1(t^{-\frac{3}{2}}x+t^{-\frac{1}{2}}v,t^{-\frac{1}{2}}v)=t^{-2d}\Gamma_1g_1(t^{-\frac{3}{2}}x,t^{-\frac{1}{2}}v),
\\\label{CC01}
&P_tf(z)
=(g_t*f)(\Gamma_tz).
\end{align}
 
\bl Let  $\bbp\in[1,\infty]^2$; $\alpha,\beta\ge0$ and  $m,n\in\mN_0$. Then, there exist finite constants  
 $C_1=C_1(d,\bbp)$ and $C_2=C_2(d,\alpha,\beta,\bbp,m,n)$ such that for every $t\in(0,1]$,
\begin{align}
    \label{CC03}
    \|g_t\|_{\bbp}\le C_1 t^{-\bba\cdot\frac{d}{2}(1-\frac{1}{\bbp})},
\end{align}
and
\begin{align}\label{CC03+}
\big\||x|^\alpha|v|^\beta|\nabla_x^m\nabla_v^n\partial_t(\Gamma_tg_t)(x,v)|\big\|_\bbp\le C_2 t^{\frac{3(\alpha- m)+(\beta-n)-2}{2}-\bba\cdot\frac{d}{2}(1-\frac{1}{\bbp})}.
\end{align}
\el
\begin{proof}
    We only show \eqref{CC03+}, \eqref{CC03} can be obtained analogously. It follows from the scaling \eqref{828:00} that
    \begin{align*}
        \partial_t(\Gamma_tg_t)(x,v)&=(-2d)t^{-2d-1}\Gamma_1g_1(t^{-3/2}x,t^{-1/2}v)+(-3/2)t^{-2d-5/2}x\cdot \nabla_x \Gamma_1g_1(t^{-3/2}x,t^{-1/2}v)\\
        &\quad+(-1/2)t^{-2d-3/2}v\cdot \nabla_v \Gamma_1g_1(t^{-3/2}x,t^{-1/2}v).
    \end{align*}
   Defining
   \begin{align*}
       G(x,v):=(-2d)\Gamma_1g_1(x,v)+(-3/2)x\cdot \nabla_x \Gamma_1g_1(x,v)+(-1/2)v\cdot \nabla_v \Gamma_1g_1(x,v),
   \end{align*}
   one sees that
   \begin{align*}
       \partial_t(\Gamma_tg_t)(x,v)=t^{-2d-1}G(t^{-3/2}x,t^{-1/2}v)
   \end{align*}
   and then
   \begin{align*}
|x|^\alpha|v|^\beta\nabla_x^m\nabla_v^n\partial_t(\Gamma_tg_t)(x,v)=t^{-2d-1+[3(\alpha-m)+(\beta-n)]/2}|t^{-3/2}x|^\alpha|t^{-1/2}v|^\beta \nabla_x^m\nabla_v^nG(t^{-3/2}x,t^{-1/2}v).
   \end{align*}
   Therefore, we have
   \begin{align*}
&\big\||x|^\alpha|v|^\beta|\nabla_x^m\nabla_v^n\partial_t(\Gamma_tg_t)(x,v)|\big\|_\bbp\\
&=t^{-2d-1+[3(\alpha-m)+(\beta-n)]/2} t^{\bba\cdot \frac{d}{\bbp}}\big\||x|^\alpha|v|^\beta|\nabla_x^m\nabla_v^nG(x,v)|\big\|_\bbp\lesssim t^{\frac{3(\alpha- m)+(\beta-n)-2}{2}-\bba\cdot\frac{d}{2}(1-\frac{1}{\bbp})},
   \end{align*}
   because $\Gamma_1g_1$ is a Schwartz function. This completes the proof.
\end{proof}

We note that for any $t\ge0$ and $\bbp\in[1,\infty]^2$, $ \|\Gamma_t f\|_{\bbp}= \| f\|_{\bbp},$ 
which shows that $(\Gamma_t)_{t\geq0}$ is invariant on $\LL^\bbp$.

\begin{lemma}
    \label{lem:est-pro-sob} 
   Let $\alpha\in (0,1)$, $\bbp=(p_x,p_v)\in[1,\infty]^2$ and put $p:=p_x\wedge p_v$. Let $(M_t(z)), z\in\mR^{2d}$ be as in \eqref{def.GM} and  $f\in\mB_{\bbp,\bba}^{\alpha}$.
   There exists constant $C=C(\alpha,p,d, \bbp)$ such  that for every $0< s\le t$ and $r\ge0$, for all $z$
   \begin{align}
       &\label{est:E-f-Z}
       \| f(M_t(z))\|_{L^p(\Omega)}\leq Ct^{-\bba\cdot\frac{d}{2\bbp}} \|f\|_{\bbp}.
   \end{align}
\end{lemma}
\begin{proof}
    We recall that $g_t$, as defined in  \eqref{def.densityG}, is  the density of  $G_t$ from \eqref{def.GM}. By H\"older’s inequality for $\bbq:=(\frac{p_x}{p_x-p},\frac{p_v}{p_v-p})$, we have
    \begin{align*}
    \| f(M_t(z))\|_{L^p(\Omega)}^p=\int_{\mR^{2d}}|f(z'+\Gamma_t z)|^pg_t(z')\dif z' \lesssim    \|f\|_{\bbp}^p\|g_t\|_{\bbq}.
    \end{align*}
    In view of \eqref{CC03}, this  yields \eqref{est:E-f-Z}.
\end{proof}




\begin{corollary}\label{lem002}
    Let $f\in \mL^{\bbp^\prime}$
  with $\bbp^\prime=(p_x^\prime,p_v^\prime)\in[1,\infty]^2$. For any $\bbp=(p_x,p_v)\in[p_x^\prime,\infty]\times[p_v^\prime,\infty]$,   any  $\delta\in[0,1]$, and  any $k\in\mN_0$, there exists constant $C>0$ depending on $\bbp^\prime,d,\bbp,\delta,k$ such that  for any $(s,t)\in(0,1]_\leq^2$,
\begin{align}
    \label{est:semi-P-L}
   \|{\nabla^k_v}(P_tf-P_s\Gamma_{t-s}f)\|_{\bbp}\le C\|f\|_{{\bbp}^\prime}|t-s|^\delta s^{d(\frac{3}{2}(\frac{1}{p_x}-\frac{1}{p_x^\prime})+\frac{1}{2}(\frac{1}{p_v}-\frac{1}{p_v^\prime})){-\frac{k}{2}}-\delta}.
 \end{align}
In particular, for $f\in \mL^\infty$, we have for any $(s,t)\in(0,1]_\leq^2$ that
\begin{align}\label{est:infty}
\|P_tf-P_s\Gamma_{t-s}f\|_{\infty}\le C\Big([(t-s)s^{-1}]\wedge1\Big)\|f\|_{\infty}.
\end{align}
\end{corollary}
 \begin{proof}
 It suffices to show \eqref{est:semi-P-L}.
 Putting $F_t:=\Gamma_tf$ and using \eqref{CC01}, one sees that
 \begin{align*}
P_tf-P_s\Gamma_{t-s}f=(\Gamma_tg_t)*(\Gamma_tf)-(\Gamma_sg_s)*(\Gamma_tf)=\int_s^t\p_r\Gamma_rg_r\ast F_t\dif r.
 \end{align*}
From here, we apply Young’s inequality for convolution with $\frac{1}{\bbp}+(1,1)=\frac{1}{\bbp^\prime}+\frac{1}{{\bbq}}$ and  \eqref{CC03+} to obtain that
 \begin{align*}
  \|  {\nabla^k_v}( P_tf-P_s\Gamma_{t-s}f)\|_{\bbp}
  &= \| {\nabla^k_v}\int_s^t\p_r\Gamma_rg_r\ast F_t\dif r\|_{\bbp}
  \lesssim \int_s^t \| {\nabla^k_v}\p_r\Gamma_rg_r\|_{\bbq}  \|F_t\|_{{\bbp}^\prime}\dif r
        \\& \lesssim\|f\|_{{\bbp}^\prime}\int_s^t r^{d(\frac{3}{2q_x}+\frac{1}{2q_v}-2){-\frac{k}{2}}-1}\dif r
\\&=\|f\|_{{\bbp}^\prime}\int_s^t r^{d(\frac{3}{2}(\frac{1}{p_x}-\frac{1}{p_x^\prime})+\frac{1}{2}(\frac{1}{p_v}-\frac{1}{p_v^\prime}))-1{-\frac{k}{2}}}\dif r
\\&\lesssim\|f\|_{{\bbp}^\prime}|t-s|^\delta s^{d(\frac{3}{2}(\frac{1}{p_x}-\frac{1}{p_x^\prime})+\frac{1}{2}(\frac{1}{p_v}-\frac{1}{p_v^\prime})){-\frac{k}{2}}-\delta}
 \end{align*}
 for any $\delta\in[0,1]$.
 This shows \eqref{est:semi-P-L}. 
\end{proof}
 We conclude the current section with some analytic estimates on $P_t$.

\begin{lemma}
     Let  $\beta\in\mR$; $\bbp_1,\bbp_2\in[1,\infty]^2$ such that $\bbp_1\leq \bbp_2$. 
     There exists a constant $C=C(d,\beta,\bbp_1,\bbp_2),$ such for any $t>0$, $h\in \bB^{\beta}_{\bbp_1;\bba}$ and $f\in \LL_{\bbp_1}$,
     \begin{gather}
           \|P_t h\|_{\bbp_2}\le C (1\wedge t)^{\frac\beta2-\frac12\bba\cdot(\frac{d}{\bbp_1}-\frac{d}{\bbp_2})}\|h\|_{\bB^{\beta}_{\bbp_1;\bba}} \tif \beta\neq \bba\cdot(\frac{d}{\bbp_1}-\frac{d}{\bbp_2}),\label{est.Ppbes}
           \\\|\nabla_vP_t h\|_{\bbp_2}\le C (1\wedge t)^{\frac\beta2-\frac12-\frac12\bba\cdot(\frac{d}{\bbp_1}-\frac{d}{\bbp_2})}\|h\|_{\bB^{\beta}_{\bbp_1;\bba}} \tif \beta\neq1+ \bba\cdot(\frac{d}{\bbp_1}-\frac{d}{\bbp_2}),\label{est.DPpbes}
           \\ \|P_t f\|_{\bbp_2}\le C(1\wedge t)^{-\frac12\bba\cdot(\frac{d}{\bbp_1}-\frac{d}{\bbp_2})}\|f\|_{\bbp_1} ,\label{est.Ppp}
           \\  \mathbf{1}_{\beta+\bba\cdot(\frac{d}{\bbp_1}-\frac{d}{\bbp_2})\neq0}\|P_tf\|_{\bB_{\bbp_2;\bba}^{\beta,1}}+ \|P_tf\|_{\bB_{\bbp_2;\bba}^{\beta}}\leq C(1\wedge t)^{-\frac\beta2-\frac12\bba\cdot(\frac{d}{\bbp_1}-\frac{d}{\bbp_2})}\|f\|_{\bbp_1}.\label{est.Pbesp}
     \end{gather}
\end{lemma}
\begin{proof}
     We put $\kappa:=\bba\cdot(\frac{d}{\bbp_1}-\frac{d}{\bbp_2})$. Let $\beta_1,\beta_2$ be two real numbers.   
    It is shown in \cite[Lemma 2.16, (2.35)]{HRZ} that for any $f\in \bB_{\bbp_1;\bba}^{\beta_1}$,
     \begin{align}\label{est:Pt-itself-1}
      \mathbf{1}_{\beta_2\neq\beta_1-\kappa_1}\|P_tf\|_{\bB_{\bbp_2;\bba}^{\beta_2,1}}+ \|P_tf\|_{\bB_{\bbp_2;\bba}^{\beta_2}}\lesssim(1\wedge t)^{-\frac{(\beta_2-\beta_1+\kappa)}{2}}\|f\|_{\bB_{\bbp_1;\bba}^{\beta_1}}.
     \end{align}
     By choosing $(\beta_1,\beta_2)=(\beta,0)$ and noting that $ \|P_t h\|_{\bbp_2}\lesssim \|P_t h\|_{\bB^{0,1}_{\bbp_2;\bba}}$ (by embedding \eqref{AB2}), we obtain \eqref{est.Ppbes} from \eqref{est:Pt-itself-1}. By \cref{lem.nablabes}, $\|\nabla_vP_t h\|_{\bbp_2}\lesssim\|\nabla_vP_t h\|_{\bB^{1,1}_{\bbp_2;\bba}}$, we can obtain \eqref{est.DPpbes} from \eqref{est:Pt-itself-1} in a similar way. 
     When $\bbp_1\neq\bbp_2$, we take $\beta=0$ in \eqref{est.Ppbes} and apply the inequality $\|f\|_{\bB^{0}_{\bbp_1;\bba}} \lesssim\|f\|_{\bbp_1} $ (by embedding \eqref{AB2}), to  obtain \eqref{est.Ppp}. In the case when $\bbp_1=\bbp_2$, \eqref{est.Ppp} is straightforward from definitions. Similarly, \eqref{est.Pbesp} is a consequence of embedding \eqref{AB2} and \eqref{est:Pt-itself-1} upon taking $(\beta_1,\beta_2)=(0,\beta)$.
\end{proof}

\subsection{Exponential estimates}
The next result allows applications of Girsanov transform. 
\begin{proposition}\label{lem:Kaha}
  Let  $\bbp_0\in[1,\infty]^2$ be such that $\bba\cdot\frac{d}{\bbp_0}<2$. Let $\|f\|_{L_T^\infty(\mL^{\bbp_0})}<\infty$.
  Let $(f_n)_n$ be a sequence of bounded functions on $[0,1]\times \R^{2d}$   such that 
  \begin{align}\label{con:Gir_gene}
\sup_n\|f_n\|_{L_T^\infty(\mL^{\bbp_0})}<\infty\quad \text{and}\quad  \lim_{h\downarrow 0} \sup_n(n^{-1}\wedge h)\|f_n\|_{L_T^\infty(\mL^\infty)}=0. 
  \end{align}
   Then for any real number  $c$,
\begin{align}\label{est:novikow-1}
    &\sup_z \mE\exp\Big(c\int_0^1f(r,M_{r}(z))\dif r\Big)<\infty,\\
   &\sup_n\sup_z \mE\exp\Big(c\int_0^1f_n(r,M_{k_n(r)}(z))\dif r\Big)<\infty.\label{est:novikow}
\end{align}
\end{proposition}

\begin{proof}
\eqref{est:novikow-1} follows from the Krylov's estimate \cite[Lemma 4.1]{RZ24} and the similar way of showing \eqref{est:novikow}. Therefore we only need to show \eqref{est:novikow}.

   Without loss of generality we can assume that  $(f_n)_n$ and $c$ are non-negative.  We fix $(s,t)\in[0,1]_\leq^2$ and put $ I_{s,t}:=    \int_s^tf_n(r,M_{k_n(r)}(z))\dif r$.
   Denoting $\overline{s}:=k_n(s)+\frac{1}{n}$, we have 
   \begin{align*}
 I_{s,t}=\big(\int_s^{\overline{s}\wedge t}+\int_{\overline{s}\wedge t}^t\big)f_n(r,M_{k_n(r)}(z))\dif r=:I_1+I_2.
   \end{align*}
   We use \eqref{con:Gir_gene} to have
   \begin{align*}
       \mE_sI_1\lesssim &\sup_{r\in[s,\bar s\wedge t]}\|f_n(r)\|_\infty (\overline{s}\wedge t-s)
       \lesssim (n^{-1}\wedge(t-s))\|f_n\|_\infty.
   \end{align*}
   When $r\ge \overline{s}$, $k_n(r)\ge s$. By Markov property and \eqref{est:E-f-Z}, we have 
   \begin{align*}
       \mE_sI_2=  \int_{\overline{s}\wedge t}^t P_{k_n(r)-s} f_n(r, M_s(z))\dif r
       \lesssim \int_{\overline{s}\wedge t}^t |k_n(r)-s|^{-\bba\cdot\frac{d}{2\bbp_0}} \|f_n\|_{L_T^\infty(\mL^{\bbp_0})}\dif r.
   \end{align*}
   It is straightforward to verify that $\int_{\overline{s}\wedge t}^t |k_n(r)-s|^{-\bba\cdot\frac{d}{2\bbp_0}}dr\lesssim (t-\overline{s})^{1-\bba\cdot\frac{d}{2\bbp_0}}\le (t-{s})^{1-\bba\cdot\frac{d}{2\bbp_0}}$ (for details, see \cite[Lemma 3.10]{LL}).
    For any $h>0$, we denote
   \begin{align*}
       osc(h):=\sup_n(n^{-1}\wedge h)\|f_n\|_\infty+h^{1-\bba\cdot\frac{d}{2\bbp_0}}\sup_n\|f_n\|_{\mL^\infty_T\mL^{\bbp_0}}.
   \end{align*}
   We have shown that there is a finite constant $\kappa$ such that $\E_sI_{s,t}\le \kappa osc(t-s)$ for every $(s,t)\in[0,1]^2_\le$.
   This implies  (see similar arguments in \cite[Theorem 2.3]{Le2022} and \cite[Lemma 3.5]{LL}) that
    for all $m\in\mN$ and $(s,t)\in[0,1]^2_{\le}$,
\begin{align*}
    \mE_sI_{s,t}^m\le m!\kappa^m osc(t-s)^m.
\end{align*}
Using \eqref{con:Gir_gene}, we can choose a $h_0>0$ such that
$
osc(h)\le\tfrac{1}{2c\kappa}
$  for all $h\le h_0$.
Then, we have for any $t-s\le h_0$, by Taylor expansion,
\begin{align*}
    \mE_s\exp(cI_{s,t})\le \sum_{m=0}^\infty 2^{-m}=2,\quad a.s.
\end{align*}
Thus, for $N:=\lfloor h_0^{-1}\rfloor$, we have
\begin{align*}
 &\mE\exp\left(c\int_0^1 f_n(r,M_{k_n(r)})dr\right)\\
 &=\mE\prod_{k=0}^{N-1}\mE_{kh_0}\exp\left(c\int_{kh_0}^{(k+1)h_0} f_n(r,M_{k_n(r)})dr\right)  \mE_{Nh_0}\exp\left(c\int_{Nh_0}^{1} f_n(r,M_{k_n(r)})dr\right)\\
 &\le 2^{N+1}.
\end{align*}
Since $N$ is independent of $n$, this completes the proof.
\end{proof}
  

\section{Proof of main results}\label{sec.proofs}

\medskip


Let $\smooth$ (resp. $C_c, C_b^\infty$)  be the space of smooth compactly supported (resp. continuous compactly supported, bounded smooth with bounded derivatives)  functions on $\R^{2d}$. 
\begin{lemma} \label{lem:app}  Let $\bbr\in[1,\infty)^2$, $\bbr'$ be the H\"older conjugate of $\bbr$ and $f$ be a function in $\LL^\bbr$. Then
\begin{align}\label{id.dense}
    \|f\|_\bbr=\sup_{g\in C_c: \|g\|_{\bbr'}= 1}\int_{\R^{2d}} f(z)g(z)dz=\sup_{g\in\smooth: \|g\|_{\bbr'}= 1}\int_{\R^{2d}} f(z)g(z)dz.
\end{align}
\end{lemma}
\begin{proof} 
Since $\smooth$ is dense in $C_c$ with respect to the norm $\|\cdot\|_{\bbr'}$, the second identity in \eqref{id.dense} follows. 
We now show the first identity in \eqref{id.dense}.  We put \[J(f)=\sup_{g\in C_c: \|g\|_{\bbr'}= 1}\int_{\R^{2d}} f(z)g(z)dz.\]
We observe that $C_c$ is dense in $\LL^\bbr$ and  that $f\mapsto J(f)$ is continuous on $\LL^\bbr$. Hence, it suffices to show $\|f\|_\bbr=J(f)$ for $f\in C_c$. We assume without loss of generality that $\|f\|_\bbr\neq0$. Define
    \begin{align}\label{def.gf}
        g_f(x,v)=\|f\|_\bbr^{1-r_v}  \1_{(\|f(\cdot,v)\|_{r_x}>0)}\|f(\cdot,v)\|_{r_x}^{r_v-r_x}|f(x,v)|^{r_x-1}.
    \end{align} 
    Using the fact $f$ is uniformly continuous, it is straightforward to verify that $g_f\in C_c$ and $ \|g_f\|_{\bbr'}\le1$. By Lusin theorem, there is a function $\psi\in C_c$ such that $\|\psi\|_\infty\le 1$, and $\psi=\sgn(f)$ except on a set $E$ such that  $\int_E|fg_f|dz<\varepsilon$. 
    We then write
    \begin{align*}
        \|f\|_\bbr=\int fg_f\sgn(f)dz=\int f g_f\psi dx+\int_E fg_f(\sgn(f)-\psi)dz.
    \end{align*}
    Noting that $g_f\psi\in C_c$ and $\|g_f\psi\|_{\bbr'}\le 1$, we obtain from the above that
    \begin{align*}
        \|f\|_\bbr\le J(f)+2 \varepsilon. 
    \end{align*}
    Since $\varepsilon$ is arbitrary, this implies that $\|f\|_\bbr\le J(f)$. On the other hand, $J(f)\le \|f\|_\bbr$ evidently by H\"older inequality, and hence, we have $\|f\|_\bbr=J(f)$. 
\end{proof}
Throughout the rest of the section,  \cref{ass:main,ass.bn} are always enforced throughout the section.
We fix $t\in(0,1]$ and $\varphi\in \smooth$ such that $\|\varphi\|_\bbq=1$. 
By applying It\^o's formula to $r\to P_{t-r}\varphi(Z_r)$ and $P_{t-r}\varphi(Z^n_r)$, one sees that 
\begin{align*}
    \mE\varphi(Z_t)=\mE P_t\varphi(Z_0)+\mE\int_0^t b(r,Z_r)\cdot \nabla_vP_{t-r}\varphi(Z_r)\dif r,
\end{align*}
and
\begin{align*}
\mE\varphi(Z_t^n)=\mE P_t\varphi(Z_0)+\mE\int_0^t \Gamma_{r-k_n(r)}b_n(r,Z_{k_n(r)}^n)\cdot \nabla_vP_{t-r}\varphi(Z_r^n)\dif r.
\end{align*}
This implies that
\begin{align}
    | \mE\varphi(Z_t)- \mE\varphi(Z^n_t)|&\le \Big|\mE\int_0^t (b-b_n)(r,Z_r)\cdot \nabla_vP_{t-r}\varphi(Z_r)\dif r\Big|\no\\
    &\quad+\Big|\mE\int_0^t b_n(r)\cdot \nabla_vP_{t-r}\varphi(Z_r)-b_n(r)\cdot \nabla_vP_{t-r}\varphi(Z_r^n)\dif r\Big|\no\\
    &\quad+\Big|\mE\int_0^t \big(b_n(r,Z^n_r)-\Gamma_{r-k_n(r)}b_n(r,Z_{k_n(r)}^n)\big)\cdot \nabla_vP_{t-r}\varphi(Z_r^n)\dif r\Big|\no\\
    &=:I_1^n(t)+I_2^n(t)+I_3^n(t).\label{0419:00}
\end{align}
In the following we estimate $I_i^n, i=1,2,3$ separately.

\subsection{Taming error}
We derive estimates for  $I_1^n$ from \eqref{0419:00}.
As an intermediate step, we have the following result.

\begin{lemma}\label{lem.lprho}
  For any $\bar\bbp,\bar\bbq\in[{\bf2},\infty]^2$ with $1/\bar\bbp+1/\bar\bbq\le \1$,  there are constants $C_i>0,i=1,2$, 
  such that for any $0<s\leq t\le 1$ and $f_1\in \mL^{\bar\bbp}$, $f_2\in \mL^{\bar\bbq}$,
\begin{align}\label{est:add-Z}
\mE |f_1(Z_{s})f_2(Z_t)|&\le C_1 s^{-\frac{1}{2}(\bba\cdot\frac{d}{\bar\bbp}+\bba\cdot\frac{d}{\bar\bbq})}\|f_1\|_{\bar\bbp}\|f_2\|_{\bar\bbq},\\
\label{0418:01}
    \mE |f_1(Z^n_{s})f_2(Z^n_t)|&\le C_2 s^{-\frac{1}{2}(\bba\cdot\frac{d}{\bar\bbp}+\bba\cdot\frac{d}{\bar\bbq})}\|f_1\|_{\bar\bbp}\|f_2\|_{\bar\bbq}.
\end{align}
In particular, there is a finite constant $C>0$ such that for any $t\in(0,1]$ 
\begin{align}
\|\rho_t\|_{\bar\bbp'}, \quad \|\rho_t^n\|_{\bar\bbp'}\leq C t^{-\frac{1}{2}(\bba\cdot\frac{d}{\bar\bbp})}\label{est:density-Zn} \twhere1/\bar\bbp'+1/\bar\bbp=(1,1).
\end{align}
\end{lemma}
\begin{proof}  
Let us first show \eqref{0418:01}. \eqref{est:add-Z} then follows in a similar and even easier manner. Let  $M_t(z)$  be defined as \eqref{def.GM}.  Let 
\begin{align*} 
   \rho=\rho^n(z):=&\exp\Big(-\int_0^1\Gamma_{s-k_n(s)}b_n(s,M_{k_n(s)}(z))\dif W_s-\frac{1}{2}\int_0^1\big|\Gamma_{s-k_n(s)}b_n(s,M_{k_n(s)}(z))\big|^2\dif s\Big).
\end{align*}
Note that $f_n(t,x):=|\Gamma_{t-k_n(s)}b_n(t,x)|^2$  satisfies \eqref{con:Gir_gene} with $\bbp_0:=\bbp/2$.
Applying \cref{lem:Kaha}, 
 we see that for any $\lambda>0$,
\begin{align*}
\sup_n\sup_z\mE\exp\Big(\lambda\int_0^1\big|\Gamma_{s-k_n(s)}b_n(s,M_{k_n(s)}(z))\big|^2\dif s\Big)<\infty.
\end{align*}
This implies that $\E \rho=1$ and $\sup_n\sup_z \E(\rho)^\lambda<\infty$ for any $\lambda>0$. By Girsanov theorem, we see that
 $Z^n(z)$ shares the same law under $\mP$ as $M(z)$ under $\mQ$, where $\dif \mQ:=\rho\dif \mP.$ 
 In particular, we have 
 \begin{align*}
      \mE |f_1(Z^n_{s})f_2(Z^n_t)|= \E_\mQ|f_1(M_{s})f_2(M_t)|.
 \end{align*}
By H\"older inequality 
\begin{align*}
    \mE |f_1(Z^n_s)f_2(Z^n_t)|&\lesssim \|f_1(M_{s})f_2(M_t)\|_{L^{2}(\Omega)}\\
    &=\left(\int_{\mR^{4d}}|f_1|^{2}(z)|f_2|^{2}(\Gamma_{t-s}z+z')g_{t-s}(z')g_{s}(z)\dif z\dif z'\right)^\frac{1}{2},
\end{align*}
which, upon applying H\"older inequality to the integral with respect to $z$, yields
\begin{align*}
    \mE |f_1(Z^n_{s})f_2(Z^n_t)|\lesssim \|f_1\|_{\bar \bbp}\|f_2\circ\Gamma_{t-s}\|_{\bar\bbq}\|g_{t-s}\|_{\mathbf{1}}\|g_{s}\|_{\bbr}\lesssim  s^{-\bba\cdot \frac{d}{2\bbr}}
\|f_1\|_{\bar\bbp}\|f_2\|_{\bar\bbq}
\end{align*}
where $1/\bbr=1/\bar\bbp+1/\bar\bbq$, which is \eqref{0418:01}.  \eqref{est:density-Zn} is simply obtained by duality and \eqref{est:add-Z}, \eqref{0418:01} with taking $f_1$ therein to be some constant function and $\bar\bbq=(\infty,\infty)$. 
\end{proof}
\begin{remark}\label{rmk.girpara}
    The condition \eqref{con.bngir}, which rules out $\kappa=1/2$ in \eqref{def.bncutoff}, is crucial for the application of Girsanov theorem in the previous proof.  It may be possible to employ parametrix method (as in \cite{JM21}), avoiding usage of the Girsanov theorem, to treat the case of \eqref{def.bncutoff} with $\kappa=1/2$. 
\end{remark}
\begin{lemma}\label{lem:estIn2}
    Recall that  $I_1^n(t)$ is defined in \eqref{0419:00}.
     For any $\bbp\leq\bbq$, $t\in(0,1]$, we have
    \begin{align}\label{0419:04}
    |I_1^n(t)|\lesssim n^{-\vartheta}t^{-\frac12(\bba\cdot\frac{d}{\bbp}+\bba\cdot\frac{d}{\bbq})}\|\varphi\|_\bbq.
\end{align}
\end{lemma}
\begin{proof}

 For any $s\in(0,t]$, we put 
    $
     h_s:=(b_n(s)-b(s))\cdot\nabla_v P_{t-s}\varphi
    $. 
    Applying It\^o's formula to $r\mapsto P_{s-r} h_s(Z_r)$, we have
\begin{align*}
    \mE h_s(Z_s)=\mE P_s h_s(Z_0)+\mE\int_0^s b(r,Z_{r})\cdot \nabla_v P_{s-r} h_s(Z_r)\dif r,
\end{align*}
which by \eqref{est:add-Z} 
implies that 
\begin{align*}
    |\mE h_s(Z_s)|\lesssim \|P_s h_s\|_\infty+ \int_0^s r^{-\bba\cdot \frac{d}{\bbp}}\|\nabla_v P_{s-r} h_s\|_{\bbp}\dif r.
\end{align*}
We note that $h$ belongs to $L_T^\infty(\mL^\bbp)$, hence, $\E\int_0^1|h_s(Z_s)|ds$ is finite in view of \eqref{est:add-Z}.
Thus, by Fubini and the previous estimate, we have 
\begin{align}
\label{est:Ephib} 
I_1^n(t)
\leq \int_0^t\|P_s h_s\|_\infty\dif s+ \int_0^t\int_0^s r^{-\bba\cdot \frac{d}{\bbp}}\|\nabla_v P_{s-r} h_s\|_{\bbp}\dif r\dif s.
\end{align}

Using Bony' paraproduct (see \cref{sec:3.2}), we decompose $h_s=h^\prec_s+h^\succcurlyeq_s$ where
\begin{align*}
    h^\prec_s:=\nabla_v P_{t-s}\varphi\prec(b_n-b)(s)
    \tand
    h^\succcurlyeq_s:=\nabla_v P_{t-s}\varphi\succcurlyeq(b_n-b)(s).
\end{align*}
We estimate each term in \eqref{est:Ephib} accordingly.
First, by \eqref{0426:09} and \cref{lem.nablabes}, we have for $\alpha\in\{0,-\delta\}$ 
\begin{align*}
     \| h^\prec_s\|_{\bB^{\alpha}_{\bbp;\bba}}
   \lesssim \|\nabla_vP_{t-s}\varphi\|_\infty\|b_n-b\|_{\bB^{\alpha}_{\bbp;\bba}}
   \lesssim \|P_{t-s}\varphi\|_{\bB^{1,1}_{\infty;\bba}}\|b_n-b\|_{\bB^{\alpha}_{\bbp;\bba}}.
\end{align*}
Applying \eqref{est.Pbesp}, we have $\|P_{t-s}\varphi\|_{\bB^{1,1}_{\infty;\bba}}\lesssim (t-s)^{-\frac12-\bba\cdot\frac{d}{2\bbq}}\|\varphi\|_\bbq $.
These estimates imply that 
\begin{align}
   \| h^\prec_s\|_{\bB^{\alpha}_{\bbp;\bba}}
   \lesssim n^{\alpha \vartheta}(t-s)^{-\frac12-\bba\cdot\frac{d}{2\bbq}}\|\varphi\|_\bbq.\label{h^prec2}
\end{align}
Together with \eqref{est.Ppbes}, this implies that 
\begin{align}\label{h^prec1}
    \|P_s h^\prec_s\|_{\mL^\infty}
    \lesssim \min( s^{-\frac\delta2-\bba\cdot\frac{d}{2\bbp}}\| h^\prec_s\|_{\bB^{-\delta}_{\bbp;\bba}},\ s^{-\bba\cdot\frac{d}{2\bbp}}\| h^\prec_s\|_{\bB^{0}_{\bbp;\bba}})\lesssim s^{-\frac12-\bba\cdot\frac{d}{2\bbp}} n^{-\vartheta}(t-s)^{-\frac12-\bba\cdot\frac{d}{2\bbq}}\|\varphi\|_\bbq.
\end{align}
Concerning the term of $ h^\succcurlyeq_s$, we define $\bbr$ via the relation $1/\bbr:=1/\bbp+1/\bbq$, then apply \eqref{est:Paraconse} and \cref{lem.nablabes} to see that
\begin{align*}
    \| h^\succcurlyeq_s\|_{\bbr}
    &\lesssim \min\left(\|\nabla_v P_{t-s}\varphi\|_{\bB^{\delta,1}_{\bbq;\bba}}\|b_n-b\|_{\bB^{-\delta,1}_{\bbp;\bba}},\ \|\nabla_v P_{t-s}\varphi\|_{\bB^{0,1}_{\bbq;\bba}}\|b_n-b\|_{\bbp} \right)
    \\&\lesssim \min\left(\| P_{t-s}\varphi\|_{\bB^{1+\delta,1}_{\bbq;\bba}}\|b_n-b\|_{\bB^{-\delta,1}_{\bbp;\bba}},\ \| P_{t-s}\varphi\|_{\bB^{1,1}_{\bbq;\bba}}\|b_n-b\|_{\bbp} \right).
\end{align*}
We use \eqref{est.Pbesp} and \cref{ass.bn} to estimate each term on the right-hand side above, this and embedding \eqref{AB2} yield 
\begin{align}
    \|h^\succcurlyeq_s\|_{\bB^0_{\bbr;\bba}}
    \lesssim \| h^\succcurlyeq_s\|_{\bbr}
    &\lesssim \|\varphi\|_\bbq(t-s)^{-\frac{1}{2}}\left([(t-s)n^{2\vartheta}]^{-\frac\delta2}\wedge 1 \right).\label{h^succ}
\end{align}
Then, applying \eqref{est.Ppp}, we have
\begin{align*}
    \|P_s h^\succcurlyeq_s\|_{\infty}&\lesssim s^{-\bba\cdot\frac{d}{2\bbr}} \| h^\succcurlyeq_s\|_{\bbr}\lesssim s^{-\bba\cdot\frac{d}{2\bbr}}\|\varphi\|_\bbq(t-s)^{-\frac{1}{2}}\left([(t-s)n^{2\vartheta}]^{-\frac\delta2}\wedge 1 \right).
\end{align*}
Together with \eqref{h^prec1} and \eqref{lem:A1}, we obtain that
\begin{align}
        \int_0^t \|P_s h_s\|_\infty\dif s&\le \int_0^t \|P_s h^\prec_s\|_\infty\dif s+\int_0^t \|P_s h^\succcurlyeq_s\|_\infty\dif s
    \nonumber\\ \nonumber
    &\lesssim \|\varphi\|_{\bbq}\Big(n^{-\vartheta}\int_0^ts^{-\frac12-\bba\cdot\frac{d}{2\bbp}}(t-s)^{-\frac12-\bba\cdot\frac{d}{2\bbq}}\dif s\\\nonumber
    &\qquad+\int_0^t s^{-\bba\cdot\frac{d}{2\bbr}}(t-s)^{-\frac12}\left([(t-s)n^{2\vartheta}]^{-\frac\delta2}\wedge 1 \right)\dif s \Big)\\
    &\lesssim \|\varphi\|_{\bbq}n^{-\vartheta}t^{-\bba\cdot\frac{d}{2\bbr}}.\label{est:Ephib1}
\end{align}
On the other hand, by \eqref{est.DPpbes}, we have
\begin{multline*}
    \|\nabla_v P_{s-r} h_s\|_{\bbp}\le\|\nabla_v P_{s-r} h^\prec_s\|_{\bbp}+\|\nabla_v P_{s-r} h^\succcurlyeq_s\|_{\bbp}\\
    \lesssim \min\left((s-r)^{-\frac{1+\delta}2}\|h^\prec_s\|_{\bB^{-\delta}_{\bbp;\bba}},\ (s-r)^{-\frac12}\|h^\prec_s\|_{\bB^{0}_{\bbp;\bba}}\right)+(s-r)^{-\frac12-\bba\cdot\frac{d}{2\bbq}}\|h^\succcurlyeq_s\|_{\bB^0_{\bbr;\bba}},
\end{multline*}
which, by \eqref{h^prec2}, \eqref{h^succ}, implies that
\begin{align*}
   \|\nabla_v P_{s-r} h_s\|_{\bbp}
   &\lesssim\|\varphi\|_\bbq(s-r)^{-\frac12}(t-s)^{-\frac12-\bba\cdot\frac{d}{2\bbq}}\left(\left([(s-r)n^{2\vartheta}]^{-\frac\delta2}\right]\wedge 1 \right)\\
   &\quad+\|\varphi\|_\bbq(s-r)^{-\frac12-\bba\cdot\frac{d}{2\bbq}}(t-s)^{-\frac{1}{2}}\left([(t-s)n^{2\vartheta}]^{-\frac\delta2}\wedge 1 \right).
\end{align*}
Using the previous estimate and \eqref{lem:A1}, we have
\begin{align}\label{est:Ephib2}
    &\int_0^t\int_0^sr^{-\bba\cdot \frac{d}{\bbp}}\|\nabla_v P_{s-r} h_s\|_{\bbp}\dif r\dif s\no\\
    &\lesssim \|\varphi\|_\bbq\Big(\int_0^t\int_0^sr^{-\bba\cdot \frac{d}{\bbp}}(s-r)^{-\frac12}(t-s)^{-\frac12-\bba\cdot\frac{d}{2\bbq}}\left(\left([(s-r)n^{2\vartheta}]^{-\frac\delta2}\right]\wedge1\right) \dif r\dif s\no\\
    &\quad+\int_0^t\int_0^s r^{-\bba\cdot \frac{d}{\bbp}}(s-r)^{-\frac12-\bba\cdot\frac{d}{2\bbq}}(t-s)^{-\frac12}\left([(t-s)n^{2\vartheta}]^{-\frac\delta2}\wedge 1 \right)\dif r\dif s\Big)\no\\
    &\lesssim \|\varphi\|_\bbq\Big(n^{-\vartheta}\int_0^ts^{-\bba\cdot\frac{d}{\bbp}}(t-s)^{-\frac12-\bba\cdot\frac{d}{2\bbq}}\dif s
    \no\\&\quad
    +\int_0^ts^{\frac12-\bba\cdot \frac{d}{\bbp}-\bba\cdot \frac{d}{2\bbq}}(t-s)^{-\frac12}\left([(t-s)n^{2\vartheta}]^{-\frac\delta2}\wedge 1 \right)\dif s\Big)\no\\\
    &\lesssim \|\varphi\|_\bbq n^{-\vartheta} t^{\frac12-\bba\cdot \frac{d}{\bbp}-\bba\cdot \frac{d}{2\bbq}}.
\end{align}
Plugging   \eqref{est:Ephib1} and \eqref{est:Ephib2} into \eqref{est:Ephib}, we obtain \eqref{0419:04}.
\end{proof}


\subsection{Quadrature error}
To estimate the quadrature error, the term $I^n_3$ in \eqref{0419:00}, we will make use of the following results.
\begin{lemma}\label{lem.technical}
    For any 
    $\bbq\in[1,\infty]^2$ with $1/\bbp+1/\bbq\le {\bf1}$, then 
    there is a constant $C>0$ such that 
    for all $t\in(n^{-1},1)$, $f\in C_b^\infty$ 
    and $h\in \mL^{\bbp}$,
    \begin{align}\label{0416:00}
         &  \left|\mE  h(Z^n_{k_n(t)})\Big(\Gamma_{t-k_n(t)}f(Z^n_{k_n(t)})-f(Z^n_t)\Big)\right|
      \nonumber\\  &\le Cn^{-1}\|h\|_{\bbp}\Big(
      (k_n(t))^{-\bba\cdot \frac{d}{\bbp}}\left(n^{-1}\|\nabla_x f\|_{\mL^\infty}+\|\nabla_v f\|_{\mL^\infty}\right)     
      \nonumber\\&\qquad\qquad\qquad\qquad\qquad\qquad+(k_n(t))^{-\frac12(\bba\cdot \frac{d}{\bbp}+\bba\cdot \frac{d}{\bbq})}(\|f\|_{\bB^{2}_{\bbq;\bba}}+\|\nabla_v^2 f\|_{\mL^\bbq})\Big).
    \end{align}
\end{lemma}
\begin{proof}
We can assume that the norms of $f$ which appears in the right-hand of \eqref{0416:00} are finite.
  We define
 \begin{align*}
     A_n(t,z):=\int_{k_n(t)}^t\int_{k_n(t)}^s\Gamma_{r-k_n(r)}b_n(r,z)\dif r\dif s,\ \ \quad \ \ B_n(t,z):=\int_{k_n(t)}^t\Gamma_{s-k_n(s)}b_n(s,z)\dif s.
 \end{align*}
Note that
\begin{align*}
&Z^n_t-Z^n_{k_n(t)}=\left(\int_{k_n(t)}^t V^n_r\dif r,B_n(t,Z^n_{k_n(t)})+W_t-W_{k_n(t)}\right)\no\\
&=\Big((t-k_n(t))V^n_{k_n(t)}+A_n(t,Z^n_{k_n(t)})+\int_{k_n(t)}^t(W_s-W_{k_n(t)})\dif s, B_n(t,Z^n_{k_n(t)})+W_t-W_{k_n(t)}\Big).
\end{align*}
Since $\Big(W_t-W_{k_n(t)},\int_{k_n(t)}^t(W_s-W_{k_n(t)})\dif s\Big)\stackrel{(d)}{=}\Big(W_{t-k_n(t)},\int_0^{t-k_n(t)}W_s\dif s\Big)=G_{t-k_n(t)}$, which is independent of $Z^n_{k_n(t)}$, recalling that $G_t$ and $Z^n_t$ have densities $g_t$ and $\rho^n_t$, respectively, we have
\begin{align*}
    &\sI_n(t):=\left|\mE h(Z^n_{k_n(t)})\Big(\Gamma_{t-k_n(t)}f(Z^n_{k_n(t)})-f(Z^n_t)\Big)\right|\\
    &=\Big|\int_{\mR^{4d}}\Big(f(x+(t-k_n(t)v,v)-f(x+(t-k_n(t)v)+A_n(t,z)+x',v+B_n(t,z)+v')\Big)\\
     &\qquad\qquad\qquad\qquad\qquad\qquad h(z)g_{t-k_n(t)}(z')\rho^n_{k_n(t)}(z)\dif z\dif z'\Big|,
\end{align*}
where $z=(x,v)$ and $z'=(x',v')$ in the integral.
We note that
\begin{align*}
    &\quad f(x+(t-k_n(t)v,v)-f(x+(t-k_n(t)v)+A_n(t,z)+x',v+B_n(t,z)+v')\\
    &=-\delta_{(A_n(t,z),B_n(t,z))}f(x+(t-k_n(t)v)+x',v+v')-\delta_{(x',v')}f(x+(t-k_n(t))v,v),
\end{align*}
where $\delta_{(x',v')}f(x,v):=f(x+x',v+v')-f(x,v)$.
 We note that
\begin{align*}
    |\delta_{(A_n(t,z),B_n(t,z))}f(x+(t-k_n(t)v)+x',v+v')|&\le |A_n(t,z)|\|\nabla _xf\|_{\mL^\infty}+|B_n(t,z)|\|\nabla_vf\|_{\mL^\infty},
\end{align*}
which implies that
\begin{align*}
    \sI_n(t)\le&\int_{\mR^{4d}}|h(z)|(|A_n(t,z)|\|\nabla _xf\|_{\mL^\infty}+|B_n(t,z)|\|\nabla_vf\|_{\mL^\infty})g_{t-k_n(t)}(z')\rho^n_{k_n(t)}(z)\dif z\dif z' \\
    &+\Big|\int_{\mR^{4d}}\Big(\delta_{(x',v')}f(x+(t-k_n(t))v,v)\Big)h(z)g_{t-k_n(t)}(z')\rho^n_{k_n(t)}(z)\dif z\dif z'\Big|\\
    =:&\sI^1_n(t)+\sI^2_n(t).
\end{align*}
Note that by \eqref{con:bnbdd},
\begin{align*}
    \|A_n(t)\|_{\bbp}\lesssim \|b_n\|_\bbp\int^t_{k_n(t)}\int^s_{k_n(s)}dr ds\le n^{-2}\quad \text{and}\quad \|B_n(t)\|_{\bbp}\lesssim n^{-1}.
\end{align*}
By H\"older inequality and heat kernel estimate of $\rho^n_t$, \eqref{est:density-Zn}, we have
\begin{align*}
    \sI^1_n(t)&\lesssim \|h\|_{\bbp}(\|A_n(t)\|_{\bbp}\|\nabla_x f\|_{\mL^\infty}+\|B_n(t)\|_{\bbp}\|\nabla_v f\|_{\mL^\infty})\|g_{t-k_n(t)}\|_1\|\rho^n_{k_n(t)}\|_{(\bbp/2)'}\\
    &\lesssim (k_n(t))^{-\bba\cdot\frac{d}{\bbp}}\|h\|_{\bbp}(n^{-2}\|\nabla_x f\|_{\mL^\infty}+n^{-1}\|\nabla_v f\|_{\mL^\infty}).
\end{align*}
For $\sI^2_n(t)$, it follows from \eqref{xy-difference}  that for any $z':=(x',v')\in\mR^{2d}$,
\begin{align*}
    \|\delta_{(x',v')}f-v'\nabla_v f\|_{\mL^{\bbq}}&\lesssim \|\delta_{(x',v')}f-\delta_{(0,v')}f\|_{\mL^{\bbq}}+\|\delta_{(0,v')}f-v'\nabla_v f\|_{\mL^{\bbq}}
    \\&\lesssim|x'|^{\frac23}\|f\|_{\bB^{2}_{\bbq;\bba}}+|v'|^2\|\nabla_v^2 f\|_{\mL^\bbq}\\
    &\lesssim |z'|_\bba^2(\|f\|_{\bB^{2}_{\bbq;\bba}}+\|\nabla_v^2 f\|_{\mL^\bbq}).
\end{align*}
Then by the symmetry, H\"older inequality 
we have 
\begin{align*}
    \sI^2_n(t)&=\Big|\int_{\mR^{4d}}\Big(\delta_{(x',v')}f(x+(t-k_n(t))v,v)-v'\cdot\nabla_vf(x+(t-k_n(t))v,v)\Big)\\
    &\qquad\qquad\qquad\qquad\qquad\qquad h(z)g_{t-k_n(t)}(z')\rho^n_{k_n(t)}(z)\dif z\dif z'\Big|\\
    &=\Big|\int_{\mR^{4d}}\Big(\delta_{(x',v')}f(x,v)-v'\cdot\nabla_vf(x,v)\Big)g_{t-k_n(t)}(z')(h\rho^n_{k_n(t)})(\Gamma_{k_n(t)-t}z)\dif z\dif z'\Big| \\
    &\lesssim \|h\|_{\bbp}(\|f\|_{\bB^{2}_{\bbq;\bba}}+\|\nabla_v^2 f\|_{\mL^\bbq})\|\rho^n_{k_n(t)}\|_{\bbr}\int_{\mR^d}|z'|_\bba^2g_{t-k_n(t)}(z')\dif z'\\
    &\lesssim (k_n(t))^{-\frac{1}{2}(\bba\cdot\frac{d}{\bbp}+\bba\cdot\frac{d}{\bbq})}\|h\|_{\bbp}n^{-1}(\|f\|_{\bB^{2}_{\bbq;\bba}}+\|\nabla_v^2 f\|_{\mL^\bbq}),
\end{align*}
where $1/\bbr=1-1/\bbp-1/\bbq$.
    This completes the proof.
\end{proof}
\begin{lemma}     \label{lem:I3}
    For and $f\in \LL^\bbp$ and any $2/n\le s\le t\le 1$  
    \begin{align}
        \label{est:I3n-6.31move}
      |\mE f(Z^n_t)-\mE \Gamma_{t-s}f(Z^n_s)|\lesssim    \|f\|_{\bbp}|t-s|^{\frac12} s^{-\frac12-\bba\cdot\frac{d}{2\bbp}}.  
    \end{align}
\end{lemma}
\begin{proof}
   We consider first the case  $f\in C_c^\infty$. 
We apply It\^o's formula to $r\to P_{t-r}f(Z^n_r)$, using the fact that $\p_r P_rf=(\Delta_v+v\cdot \nabla_x)P_rf$, to obtain that
\begin{align*}
    \mE f(Z^n_t)=\mE P_{t}f(Z_0)+\mE\int_0^t \Gamma_{r-k_n(r)}b_n(r,Z^n_{k_n(r)})\cdot \nabla_v P_{t-r}f(Z^n_r)\dif r.
\end{align*}
For any $0<s<t$, by replacing $f$ with $\Gamma_{t-s}f$, we also have 
\begin{align*}
    \mE \Gamma_{t-s}f(Z^n_{s})=\mE P_{s}\Gamma_{t-s}f(Z_0)+\mE\int_0^{s} \Gamma_{r-k_n(r)}b_n(r,Z^n_{k_n(r)})\cdot \nabla_v P_{s-r}\Gamma_{t-s}f(Z^n_r)\dif r.
\end{align*}
Hence, we have
\begin{align}\label{eq:I-123}
    &|\mE f(Z^n_t)-\mE \Gamma_{t-s}f(Z^n_{s})|
    \\&\le \|P_{t}f-P_{s}\Gamma_{t-s}f\|_\infty+\mE\int_s^t |\Gamma_{r-k_n(r)}b_n(r,Z^n_{k_n(r)})\cdot \nabla_v P_{t-r}f(Z^n_r)|\dif r
    \nonumber\\
    &\quad+\mE\int_0^{s}|\Gamma_{r-k_n(r)}b_n(r,Z^n_{k_n(r)})\nabla_v(P_{t-r}-P_{s-r}\Gamma_{t-s})f(Z^n_r)|\dif r
    =:H_1+H_2+H_3.\nonumber
\end{align}
For $H_1$, taking $(\delta,k,\bbp,\bbp')=(\frac12,0,\infty,\bbp)$ in \eqref{est:semi-P-L}, we have 
\begin{align}
    \|P_{t}f-P_{s}\Gamma_{t-s}f\|_\infty\lesssim\|f\|_{\bbp}|t-s|^{\frac12} s^{-\frac12-\bba\cdot\frac{d}{2\bbp}}.
\end{align}

Applying \eqref{0418:01}, \eqref{est.DPpbes}, \eqref{AB2} and the fact that $k_n(r)\ge k_n(s)\ge s/2$,  it follows that
\begin{align}
    H_2&\lesssim \|b_n\|_{L_T^\infty(\mL^{\bbp})}\int_s^t (k_n(r))^{-\bba\cdot \frac{d}{\bbp}} \|\nabla_v P_{t-r}f\|_{\bbp}\dif r
    \nonumber\\&\lesssim s^{-\bba\cdot \frac{d}{\bbp}}\|f\|_{\bB^0_{\bbp;\bba}}\int_s^t (t-r)^{-\frac{1}{2}}\dif r\lesssim s^{-\bba\cdot \frac{d}{\bbp}}|t-s|^{\frac12}\|f\|_{\bbp}.\nonumber
\end{align}
To estimate $H_3$, we decompose the integral on $[0,s]$ into regions $[0,n^{-1}]$ and $[n^{-1},s]$. Applying \eqref{est:semi-P-L} and \eqref{con.taming},  we have 
\begin{align*}
    &\mE\int_0^{n^{-1}}|\Gamma_{r-k_n(r)}b_n(r,Z^n_{k_n(r)})\nabla_v(P_{t-r}-P_{s-r}\Gamma_{t-s})f(Z^n_r)|\dif r
    \\&\le \sup_{t\in[0,n^{-1}]}\|b_n(t,\cdot)\|_\infty\int_0^{n^{-1}}\mE|\nabla_v(P_{t-r}-P_{s-r}\Gamma_{t-s})f|(Z^n_r)\dif r
    \\&\lesssim \|f\|_{\bbp}|t-s|^{\frac12} n^\zeta\int_0^{n^{-1}}(s-r)^{-1}r^{-\bba\cdot\frac{d}{2\bbp}}\dif r
    \lesssim |t-s|^{\frac12}s^{-\frac12-\bba\cdot\frac{d}{2\bbp}}  \|f\|_{\bbp}.
\end{align*}
To obtain the last inequality above, we have used the estimate  $(s-r)^{-1}\lesssim s^{{-\frac12-\bba\cdot\frac{d}{2\bbp}}}(\frac{1}{n}-r)^{-\frac12+\bba\cdot\frac{d}{2\bbp}}$ for each $r\in[0,n^{-1}]$ and the assumption that $\zeta\le1/2$.
For the other integration over the other region, we apply \eqref{0418:01}, \eqref{est:semi-P-L}, inequality $k_n(r)\ge r/2$ for any $r>n^{-1}$ and \eqref{lem:A1} to get that
\begin{align*}
  &\mE\int_{n^{-1}}^s|\Gamma_{r-k_n(r)}b_n(r,Z^n_{k_n(r)})\nabla_v(P_{t-r}-P_{s-r}\Gamma_{t-s})f(Z^n_r)|\dif r
  \\&\le
  \|b_n\|_{L_T^\infty(\mL^{\bbp})} \int_{n^{-1}}^s (k_n(r))^{-\bba\cdot \frac{d}{\bbp}}\|\nabla_v(P_{t-r}-P_{s-r}\Gamma_{t-s})f\|_{\bbp}\dif r
  \\&\lesssim\|f\|_{\bbp}\int_{0}^s r^{-\bba\cdot\frac{d}{\bbp}}(s-r)^{-\frac{1}{2}}[((t-s)(s-r)^{-1})\wedge 1]\dif r
  \lesssim \|f\|_{\bbp}|t-s|^{\frac12} s^{-\bba\cdot\frac{d}{2\bbp}}.
\end{align*}
Combining all the estimates above, we obtain  \eqref{est:I3n-6.31move} for $f\in \smooth$

If $\bbp<(\infty,\infty)$, then \eqref{est:I3n-6.31move} also holds for $f\in\LL^\bbp$ by a density argument. For general $\bbp\in[2,\infty]^2$, we reason in the following way.
By a density argument, one can easily extend \eqref{est:I3n-6.31move} for $f\in C_b\cap \LL^p$. Next, we consider the case when $f\in \LL^\infty\cap\LL^p$. Denote the left-hand side of \eqref{est:I3n-6.31move} by $\Lambda(f)$.  We take a sequence $(f_j)$ in $ C_b $ which converges to $f$ a.e. and satisfies $\|f_j\|_\bbp\le \|f\|_\bbp$ and $\|f_j\|_\infty\le \|f\|_\infty$ for each $j$. 
Since for each $j$, $f_j\in C_b$, our considerations in the  previous case show that
\[|\Lambda(f_j)|\lesssim \|f_j\|_\bbp|t-s|^{\frac{1}{2}}s^{-\frac12-\bba\cdot\frac{d}{2\bbp}}\lesssim \|f\|_\bbp|t-s|^{\frac{1}{2}}s^{-\frac12-\bba\cdot\frac{d}{2\bbp}}.\]
Using dominated convergence theorem $\lim_j\Lambda(f_j-f)=0$. In view of the decomposition $\Lambda(f)=\Lambda(f_j)+\Lambda(f_j-f)$, these facts imply that \eqref{est:I3n-6.31move} holds for such $f$. 
Finally, for a general $f\in\LL^\bbp$, we define $f^N$ by truncating $f$ at level $N$. Then $\lim_Nf^N=f$ a.e. and $|f^N|\le |f|$ a.e.. The term $\Lambda(f^N)$ can be treated analogously as previously. The term $\Lambda(f^N-f)$ vanishes by dominated convergence theorem, which is justified by \eqref{est:add-Z}. We conclude that \eqref{est:I3n-6.31move} holds for $f\in\LL^\bbp$. 
\end{proof}
\begin{lemma}\label{lem:est-I3n}
 For any $\bbq\ge \bbp$, $t\in(0,1]$
\begin{align}\label{0419:02} 
    I^n_3(t)\lesssim n^{-\frac12}\left(t^{-\bba\cdot\frac{d}{2\bbq}-\zeta}+t^{-\frac12(\bba\cdot\frac{d}{\bbq}+\bba\cdot \frac{d}{\bbp})}\right)\|\varphi\|_\bbq,
\end{align}
where $I^n_3(t)$ is defined in \eqref{0419:00}.
\end{lemma}

\begin{proof}
We first make an observation that taking $\beta=0$ in \eqref{est.DPpbes} and using the embedding $\LL_\bbq\hookrightarrow \bB^{0}_{\bbq;\bba}$ (see \eqref{AB2}), we have 
\begin{align}\label{tmp.Pphi}
    \|\nabla_v P_{t-r}\varphi\|_\infty\lesssim (t-r)^{-\frac12-\bba\cdot \frac{d}{2\bbq}}\|\varphi\|_{\bbq} 
    \tand 
    \|\nabla_v P_{t-r}\varphi\|_\bbq\lesssim (t-r)^{-\frac12}\|\varphi\|_{\bbq}.
\end{align}

In the case when $t\le 3/n$, using \cref{ass.bn} and applying \eqref{0418:01} and \eqref{tmp.Pphi},  we obtain that
\begin{align*}
    I^n_3(t)&\le \sup_{s\le n^{-1}}\|b_n(s)\|_\infty\int_0^{t\wedge n^{-1}} \|\nabla_v P_{t-r}\varphi\|_\infty dr+\int_{t\wedge n^{-1}}^{t} (k_n(r))^{-\bba\cdot\frac{d}{2\bbp}}\|b_n(r)\|_\bbp\|\nabla_v P_{t-r}\varphi\|_\infty dr \\
    &\lesssim n^{\zeta}\int_0^t (t-r)^{-\frac12-\frac{\bba\cdot d/\bbq}{2}}\|\varphi\|_{\bbq} \dif r+\int_{t\wedge n^{-1}}^{t} (k_n(r))^{-\bba\cdot\frac{d}{2\bbp}}(t-r)^{-\frac12-\frac{\bba\cdot d/\bbq}{2}}\|\varphi\|_\bbq dr\\
    &\lesssim n^{\zeta}t^{\frac12-\bba\cdot \frac{d}{2\bbq}}\|\varphi\|_{\bbq}+\int_{t\wedge n^{-1}}^{t} r^{-\bba\cdot\frac{d}{2\bbp}}(t-r)^{-\frac12-\frac{\bba\cdot d/\bbq}{2}}\|\varphi\|_\bbq dr\\
    &\lesssim \left(n^{\zeta}t^{\frac12-\bba\cdot \frac{d}{2\bbq}}+t^{\frac12-\frac12(\bba\cdot\frac{d}{\bbp}+\frac{d}{\bbq})}\right)\|\varphi\|_{\bbq}.
   \end{align*}
   Since $n\le 3t^{-1}$, this implies \eqref{0419:02}.

In the case when $t>3/n$,  we  have
\begin{align}\label{eq:In3-S1S2}
    I^n_3(t)&\le \Big|\mE\int_0^t \Gamma_{s-k_n(s)}b_n(s,Z^n_{k_n(s)})\cdot\left(\Gamma_{s-k_n(s)}\nabla_v P_{t-s}\varphi(Z^n_{k_n(s)})-\nabla_v P_{t-s}\varphi(Z^n_s)\right)\dif s\Big|\nonumber\\
    &\quad+\Big|\mE\int_0^t \Gamma_{s-k_n(s)}(b_n(s)\cdot \nabla_v P_{t-s}\varphi)(Z^n_{k_n(s)})-(b_n(s)\cdot \nabla_v P_{t-s}\varphi)(Z^n_s)\dif s\Big|\nonumber\\
    &=:S_{1}+S_{2}.
\end{align}
We define $\tfrac{1}{\bbr}:=\tfrac{1}{\bbp}+\tfrac{1}{\bbq}$ and
\begin{align*}
    Q_n(s):=\mE \Gamma_{s-k_n(s)}b_n(s,Z^n_{k_n(s)})\cdot\left(\Gamma_{s-k_n(s)}\nabla_v P_{t-s}\varphi(Z^n_{k_n(s)})-\nabla_v P_{t-s}\varphi(Z^n_s)\right).
\end{align*}
We have
\begin{align*}
    S_{1}=\int_0^t|Q_n(s)|\dif s 
    \le\big(\int_0^{\frac{1}{n}}+\int_{\frac{1}{n}}^{t}
    \big) 
    \big|Q_n(s)\big|\dif s
    =:S_{11}+S_{12}.
\end{align*}
Using \eqref{est.DPpbes} (with $\beta=0$) and \eqref{AB2}, we have
\begin{align*}
   S_{11}\le 2n^{\zeta} \int_0^{\frac1n}\|\nabla_v  P_{t-s}\varphi \|_\infty ds \lesssim n^{\zeta} \int_0^{\frac1n}(t-s)^{-\frac12-\frac{d}{2\bbq}}\|\varphi\|_{\bbq}ds.
\end{align*}
Noting that $n^{-1}\le t/3$ and $(t-s)^{-1}\le\frac{3}{2}t^{-1}$, we have 
\begin{align*}
   S_{11}\lesssim n^{\zeta} \int_0^{\frac1n}t^{-\frac12-\frac{d}{2\bbq}}\|\varphi\|_{\bbq}ds\lesssim n^{\zeta-1}t^{-\frac12-\frac{d}{2\bbq}}\|\varphi\|_{\bbq}\lesssim n^{-\frac12}t^{-(\zeta+\frac{d}{2\bbq})}\|\varphi\|_{\bbq} .
\end{align*}
Concerning $S_{12}$, 
on one hand,  when $s>1/n$, $k_n(s)>s/2$ so that by \eqref{0418:01} and \eqref{tmp.Pphi}, we have
\begin{align}\label{0426:03}
    |Q_n(s)|\lesssim 2(k_n(s))^{-\bba\cdot\frac{d}{2\bbr}}\|b_n\|_{L_T^\infty(\LL^\bbp)}\|\nabla_vP_{t-s}\varphi\|_\bbq\lesssim s^{-\bba\cdot\frac{d}{2\bbr}}(t-s)^{-\frac12}\|\varphi\|_\bbq.
\end{align}
On the other hand, by \eqref{0416:00}, \cref{lem.nablabes}, \eqref{est.Pbesp},  we have 
\begin{align*}
|Q_n(s)|\lesssim&\|b_n\|_{L_1^\infty(\LL^\bbp)}\Big(n^{-2}(k_n(s))^{-\bba\cdot\frac{d}{\bbp}}\|\nabla_v P_{t-s}\varphi\|_{\bB^{3,1}_{\infty;\bba}}+n^{-1}(k_n(s))^{-\bba\cdot\frac{d}{\bbp}}\|\nabla_v P_{t-s}\varphi\|_{\bB^{1,1}_{\infty;\bba}}\\
&\qquad\qquad\quad+n^{-1}(k_n(s))^{-\bba\cdot\frac{d}{2\bbr}}\|\nabla_v P_{t-s}\varphi\|_{\bB^{2,1}_{\bbq;\bba}}\Big)\\
\lesssim &\|\varphi\|_\bbq \Big(n^{-2}s^{-\bba\cdot\frac{d}{\bbp}}(t-s)^{-2-\frac{d}{2\bbq}}+n^{-1}s^{-\bba\cdot\frac{d}{\bbp}}(t-s)^{-1-\frac{d}{2\bbq}}+n^{-1}s^{-\bba\cdot\frac{d}{2\bbr}}(t-s)^{-\frac32}\Big).
\end{align*}
Combining the above estimate with \eqref{0426:03}, noting that $\bba\cdot\frac{d}{2\bbq}<\frac12$ and $\bbq\ge\bbp$, we have 
\begin{align*}
   |Q_n(s)|\lesssim \|\varphi\|_\bbq&\Big( s^{-\bba\cdot\frac{d}{2\bbr}}(t-s)^{-\frac12}\min(n^{-1}(t-s)^{-1},1)+ s^{-\bba\cdot\frac{d}{\bbp}}(t-s)^{-\frac12}\big(\min(n^{-2}(t-s)^{-2},1)\\&+\min(n^{-1}(t-s)^{-1},1)\big)\Big)
   \\\lesssim \|\varphi\|_\bbq&\Big( s^{-\bba\cdot\frac{d}{2\bbr}}(t-s)^{-\frac12}\min(n^{-1}(t-s)^{-1},1)+ s^{-\bba\cdot\frac{d}{\bbp}}(t-s)^{-\frac12}\min(n^{-1}(t-s)^{-1},1)\Big),
\end{align*}
which by  \cref{lem.singint} implies that
\begin{align*}
    S_{12}&\le\int_{1/n}^t|Q_n(s)|\dif s\\&\lesssim \|\varphi\|_\bbq \int_{1/n}^t \Big( s^{-\bba\cdot\frac{d}{2\bbr}}(t-s)^{-\frac12}\min(n^{-1}(t-s)^{-1},1)+ s^{-\bba\cdot\frac{d}{\bbp}}(t-s)^{-\frac12}\min(n^{-1}(t-s)^{-1},1)\Big)\dif s\\
    &\lesssim n^{-\frac12} t^{-\bba\cdot\frac{d}{2\bbr}}\|\varphi\|_\bbq.
\end{align*}
Therefore, combining the estimate for $S_{11}$, we 
have
\begin{align}\label{est:In3-S1}
    S_1\lesssim n^{-\frac12}\left(t^{-\bba\cdot\frac{d}{2\bbq}-\zeta}+t^{-\frac12(\bba\cdot\frac{d}{\bbq}+\bba\cdot \frac{d}{\bbp})}\right)\|\varphi\|_\bbq.
\end{align}

Next we estimate $S_{2}$. By dividing the interval $[0,t]$ into $[0,1/n]\cup [1/n,2/n] \cup [2/n,t]$,  following from \eqref{0418:01}, we have 
\begin{align*}
    S_2\lesssim& \int_0^{\frac1n}\|b_n(s)\|_\infty\|\nabla_v P_{t-s}\varphi\|_\infty ds +\int_{\frac1n}^{\frac2n} (k_n(s))^{-\bba\cdot\frac{d}{2\bbp}}\|b_n(s)\|_\bbp \|\nabla_v P_{t-s}\varphi\|_\infty ds\\
    &+\int_{\frac{2}{n}}^t |\mE \Gamma_{s-k_n(s)}f(s,Z^n_{k_n(s)})-\mE f(s,Z^n_{s})|ds,
\end{align*}
where $f(s):=b_n(s)\cdot\nabla_v P_{t-s}\varphi$. 
Then using conditions \eqref{con.taming}, \eqref{con.bnrate}, inequalities \eqref{tmp.Pphi} and \eqref{est:I3n-6.31move}, we obtain that 
\begin{align*}
    S_{2}&\lesssim n^{\zeta}\int_0^{\frac1n}(t-s)^{-\frac12-\bba\cdot\frac{d}{2\bbq}}\dif s\|\varphi\|_\bbq+\int_{\frac1n}^{\frac2n} (k_n(s))^{-\bba\cdot \frac{d}{2\bbp}}(t-s)^{-\frac12-\bba\cdot\frac{d}{2\bbq}}\dif s\|\varphi\|_\bbq\\
    &\quad+\int_{\frac2n}^t |s-k_n(s)|^\frac12(k_n(s))^{-\frac12-\bba\cdot\frac{d}{2\bbp}}\|b_n(s)\cdot \nabla_v P_{t-s}\varphi\|_\bbp ds.
\end{align*}    
Noting that $(t-s)^{-1}\lesssim t^{-1}$ for $s<\frac2n<\frac{2t}{3}$, $k_n(s)\ge \frac{s}{2}$ for $s>\frac1n$, and 
$$
\|b_n(s)\cdot \nabla_v P_{t-s}\varphi\|_\bbp\le \|b_n(s)\|_\bbp\|\nabla_v P_{t-s}\varphi\|_\infty\lesssim (t-s)^{-\frac12-\bba\cdot\frac{d}{2\bbq}}\|\varphi\|_\bbq, $$ 
one sees that
        \begin{align*}
       S_{2}&\lesssim \left(n^{\zeta-1}t^{-\frac12-\bba\cdot\frac{d}{2\bbq}}+\int_{\frac1n}^{\frac2n} s^{-\bba\cdot \frac{d}{2\bbp}}t^{-\frac12-\bba\cdot\frac{d}{2\bbq}}\dif s+n^{-\frac12}\int_{\frac2n}^t s^{-\frac12-\bba\cdot\frac{d}{2\bbp}}(t-s)^{-\frac12-\bba\cdot\frac{d}{2\bbq}}ds \right)\|\varphi\|_\bbq \\
       &\lesssim \left(n^{\zeta-1}t^{-\frac12-\bba\cdot\frac{d}{2\bbq}}+n^{-1+\bba\cdot \frac{d}{2\bbp}}t^{-\frac12-\bba\cdot\frac{d}{2\bbq}}+n^{-\frac12}t^{-\frac12(\bba\cdot\frac{d}{\bbp}+\bba\cdot\frac{d}{\bbq})} \right)\|\varphi\|_\bbq\\
    &\lesssim n^{-\frac12}\left(t^{-\zeta-\bba\cdot\frac{d}{2\bbq}}+t^{-\frac12(\bba\cdot\frac{d}{\bbq}+ \frac12\bba\cdot \frac{d}{\bbp})}\right)\|\varphi\|_\bbq,
\end{align*}
provided that $n^{-1}\lesssim t$. These estimates and \eqref{est:In3-S1} yield \eqref{0419:02}.
\end{proof}

\subsection{Proof of \cref{thm-weak}}

We will need the following variation of Gr\"onwall inequality, inspired by \cite[Lemma 2.2]{Zhang10}.
 \begin{lemma}
        [Gr\"onwall  inequality] 
        \label{lem:Vol}  Let $T,\alpha,\beta,c_0$ be non-negative numbers such that $\alpha+\beta<1$. For each $0< s\le t\le T$, define $\kappa_0(t,s)=c_0(t-s)^{-\alpha}s^{-\beta}$ and for each integer $k\ge1$, $\kappa_k(t,s)=\int_s^t \kappa_0(t,u)\kappa_{k-1}(u,s)du$. Let $f,g:(0,T]\rightarrow\mR_+$ be measurable functions. Suppose there exist constants $j\ge0$ and $c_0\ge0$ such that
        \begin{align}\label{con.Vol-1} 
            \sup_{t\in(0,T]}\int_0^t\kappa_j(t,s)f(s)ds<\infty,
        \end{align}
        and
        \begin{align}
         \label{con.Vol-2}  &  f(t)\leq g(t)+\int_0^t\kappa_0(t,s)f(s)ds \quad \forall t\in(0,T]. 
        \end{align}
        Then, there exists a finite positive constant $ C=C(T,\alpha,\beta,c_0)$ such that
        \begin{align*}
            f(t)\leq g(t)+\sum_{k=0}^{j-1}\int_0^t \kappa_k(t,s)g(s)ds+ C \sup_{\tau\le t }\int_0^\tau\kappa_j(\tau,s)g(s)ds \quad \forall t\in(0,T]. 
        \end{align*}
    \end{lemma}
    \begin{proof}  
        Let $\lambda$ be a posstive constant chosen later and  define   $m_u=\sup_{t\le u}e^{-\lambda t}\int_0^t\kappa_j(t,s)f(s)ds $ which is finite for each $u\in(0,T]$  by \eqref{con.Vol-1}. Note that $u\mapsto m_u$ is non-decreasing and that 
        \begin{align*}
           \int_0^u \kappa_j(u,t) \left(\int_0^t\kappa_0(t,s)f(s)ds\right)dt
           &=\int_0^u \kappa_0(u,t) \left(\int_0^t\kappa_j(t,s)f(s)ds\right)dt
           \\&\le \int_0^u \kappa_0(u,t)e^{\lambda t}m_udt.
        \end{align*}
        Together with \eqref{con.Vol-2}, we obtain that
        \begin{align*}
            e^{-\lambda u}\int_0^u \kappa_j(u,t)f(t)dt\le e^{-\lambda u}\int_0^u \kappa_j(u,t)g(t)dt+m_u\int_0^u \kappa_0(u,t)e^{-\lambda(u-t)}dt  .
        \end{align*}
        Noting that $e^{-\lambda(u-t)}\le\min((\lambda(u-t))^{-\gamma},1)$ for any $\gamma>1-\alpha$, we apply \eqref{lem:A1} to see that 
        \begin{align*}
            \int_0^u \kappa_0(u,t)e^{-\lambda(u-t)}dt\lesssim\min (\lambda^{1-\alpha}t^{-\beta}, t^{1-\alpha-\beta})\lesssim \lambda^{\alpha+\beta-1}.
        \end{align*}
        It follows that
        \begin{align*}
            e^{-\lambda u}\int_0^u \kappa_j(u,t)f(t)dt\lesssim e^{-\lambda u}\int_0^u \kappa_j(u,t)g(t)dt+\lambda^{\alpha+\beta-1} m_u .
        \end{align*}
        By choosing $\lambda$ sufficiently large, we obtain that 
        \begin{align*}
            m_u\lesssim\sup_{t\le u} e^{-\lambda t} \int_0^t \kappa_j(t,s)g(s)ds.
        \end{align*}
        and hence, there is a constant $C=C(T,\alpha,\beta,c_0)$ so that
        \begin{align}\label{tmp.fj}
            \int_0^t\kappa_j(t,s)f(s)ds\le e^{\lambda t}m_t\le C \sup_{\tau\le t}\int_0^\tau \kappa_j(\tau,s)g(s)ds.
        \end{align}
        For each $k$, define $f_k(t)=\int_0^t \kappa_k(t,s)f(s)ds$ and similarly for $g_k$.
        From \eqref{con.Vol-2}, we have $f_{k-1}(t)\le g_{k-1}(t)+f_k(t)$ for every integer $k\ge1$, and hence, by iteration, 
        \begin{align*}
            f_0(t)\le  \sum_{k=0}^{j-1} g_{k}(t)+f_j(t).
        \end{align*}
        Combining with \eqref{tmp.fj}, we obtain that 
        \begin{align*}
            f_{0}(t)\le \sum_{k=0}^{j-1} g_{k}(t)+C\sup_{\tau\le t}g_j(\tau).
        \end{align*}
        Applying this estimate to the right-hand side of \eqref{con.Vol-2}, we obtain the result. 
    \end{proof}
\begin{proof}[Proof of  \cref{thm-weak}]

    We hinge on the inequality \eqref{0419:00}.
    Let $\bbp'$ be the H\"older conjugate of $\bbp$.
  Using  \eqref{est.DPpbes}, it is easy to see that
\begin{align}
    I^n_2(t)&=\Big|\int_0^t \<b_n(r)\cdot\nabla_vP_{t-r}\varphi,\rho_r-\rho^n_r\>\dif r\Big|\le \|b_n\|_{\bbp}\int_0^t \|\nabla_vP_{t-r}\varphi\|_\infty\|\rho_r-\rho^n_r\|_{\bbp'}\dif r\no\\
    &\lesssim \|\varphi\|_{\bbq}\int_0^t (t-r)^{-\frac{1}{2}-\bba\cdot\frac{d}{2\bbq}}\|\rho_r-\rho^n_r\|_{\bbp'}\dif r.\no
\end{align}
The terms $I^n_1$ and $I^n_3$ are estimated by  \eqref{0419:04} and \eqref{0419:02} respectively.
Applying these estimates in  \eqref{0419:00} yields that 
\begin{align*}
    \int \varphi(\rho_t-\rho^n_t)dz\lesssim n^{-(\frac12\wedge \vartheta)}\left(t^{-\zeta-\bba\cdot \frac{d}{2\bbq}}+t^{-\frac12(\bba\cdot\frac{d}{\bbp}+\bba\cdot\frac{d}{\bbq})}\right)+\int_0^t (t-r)^{-\frac{1}{2}-\bba\cdot\frac{d}{2\bbq}}\|\rho_r-\rho^n_r\|_{\bbp'}\dif r
\end{align*}
for every $\varphi\in\smooth$ such that $\|\varphi\|_{\bbq}=1$.
An application of \cref{lem:app} (with $\bbr=\bbq'$) yields 
\begin{align}
     \|\rho_t-\rho^n_t\|_{\bbq'}&\lesssim n^{-(\frac12\wedge \vartheta)}t^{-(\zeta\vee \bba\cdot \frac{d}{2\bbp})-\bba\cdot \frac{d}{2\bbq}}
     +\int_0^t (t-r)^{-\frac{1}{2}-\bba\cdot\frac{d}{2\bbq}}\|\rho_r-\rho^n_r\|_{\bbp'}\dif r.\label{tmp.rr}
\end{align}
In particular, choosing $\bbq=\bbp$, we have for every $t\in[0,1]$,
\begin{align}
     \|\rho_t-\rho^n_t\|_{\bbp'}&\le  cn^{-(\frac12\wedge \vartheta)}t^{-(\zeta\vee \bba\cdot \frac{d}{2\bbp})-\bba\cdot \frac{d}{2\bbp}}
     +c\int_0^t (t-r)^{-\frac{1}{2}-\bba\cdot\frac{d}{2\bbp}}\|\rho_r-\rho^n_r\|_{\bbp'}\dif r\label{tmp.rrp}
\end{align}
for some finite positive constant $c$. 
Define $\kappa_0(t,s)=c(t-s)^{-\frac{1}{2}-\bba\cdot\frac{d}{2\bbp}}$, $\kappa_k$ as in \cref{lem:Vol} and  $g(t)=cn^{-(\frac12\wedge \vartheta)}t^{-(\zeta\vee \bba\cdot \frac{d}{2\bbp})-\bba\cdot \frac{d}{2\bbp}}$. It is straightforward to verify that  for every integer $k\ge0$,
\begin{align*}
    \int_0^t \kappa_k(t,s)g(s)ds\lesssim 
   t^{(k+1)\gamma} g(t)
    \twhere \gamma=\frac{1}{2}-\bba\cdot\frac{d}{2\bbp}>0.
\end{align*}
Furthermore, by  \eqref{est:density-Zn} we have
\begin{align*}
  f(s):=\|\rho_s-\rho^n_s\|_{\bbp'}\leq\|\rho_s\|_{\bbp'}+\|\rho^n_s\|_{\bbp'}\lesssim s^{-\frac{1}{2}\bba\cdot\frac{d}{\bbp}}.  
\end{align*}
This implies that for every integer $k\ge0$,
\begin{align*}
    \int_0^t \kappa_k(t,s) f(s) ds\lesssim cn^{-(\frac12\wedge \vartheta)}t^{(k+1)\gamma -\bba\cdot \frac{d}{2\bbp}} .
\end{align*}
It follows that such function $f$ satisfies \eqref{con.Vol-1} with an integer $j$ such that $(j+1)\gamma -\bba\cdot \frac{d}{2\bbp}>0$. Applying \cref{lem:Vol},  \eqref{tmp.rrp} implies that
\begin{align*}
     \|\rho_t-\rho^n_t\|_{\bbp'}\lesssim g(t) .
\end{align*}
Plugging this estimate into \eqref{tmp.rr} yields \eqref{est:thm-weak-S-1}, completing the proof.
\end{proof}



\appendix
\section{Some technical lemmas}\label{app}
The following elementary lemma is similar as \cite[Lemma A.2]{HRZ}.
\begin{lemma}\label{lem.singint}
    Let $\varepsilon>0$, $\alpha_1,\alpha_2\in[0,1)$; $\gamma_1,\gamma_2\ge0$ be some fixed numbers. For any $\alpha,\gamma\ge0$, define 
    \begin{align}
        \ell_{\alpha,\gamma}(\varepsilon)=
        \left\{
        \begin{aligned}
            &\varepsilon \wedge 1 &\quad\text{ if }\quad \gamma<1-\alpha,
            \\ &\left(\varepsilon(1+|\log\varepsilon| )\right)\wedge 1&\quad\text{ if }\quad \gamma=1-\alpha,
            \\&\varepsilon^{\frac{1-\alpha}{\gamma}} \wedge 1&\quad\text{ if }\quad \gamma>1-\alpha.
        \end{aligned}
        \right.
    \end{align}
    There there is a finite constant $C=C(\alpha_1,\alpha_2,\gamma_1,\gamma_2)$ such that for all $t>0$,
    \begin{align*}
        \int_0^t s^{-\alpha_1}(t-s)^{-\alpha_2}\min(\varepsilon s^{-\gamma_1}(t-s)^{-\gamma_2},1)\le C t^{1-\alpha_1-\alpha_2} (\ell_{\alpha_1,\gamma_1}(\varepsilon t^{-\gamma_1-\gamma_2} )+\ell_{\alpha_2,\gamma_2}(\varepsilon t^{-\gamma_1-\gamma_2})).
    \end{align*}
In particular, when $\gamma_1=0$,  $\gamma_2=\gamma>1-\alpha_2$ and $\eps=\lambda^{-\gamma}$, we have
\begin{align}\label{lem:A1}
        \int_0^t s^{-\alpha_1}(t-s)^{-\alpha_2}\left([(\lambda(t-s))^{-\gamma}]\wedge 1\right)ds\le C \left(\lambda^{-1+\alpha_2}t^{-\alpha_1}\right)\wedge t^{1-\alpha_1-\alpha_2}.
    \end{align}
\end{lemma}

\begin{proof}
    Let $ I(t,\varepsilon)$ denote the integral on the left-hand side. 
    By a change of variable, we have
    \begin{align*}
        I(t,\varepsilon)=t^{1-\alpha_1-\alpha_2}I(1,\varepsilon t^{-\gamma_1-\gamma_2}).
    \end{align*}
    Thus, it suffices to estimate $I(1,\varepsilon)$ for $\varepsilon>0$. We have 
    \begin{align*}
        I(1,\varepsilon)&=\left(\int_0^{1/2}+\int_{1/2}^1\right)s^{-\alpha_1}(1-s)^{-\alpha_2}\min(\varepsilon s^{-\gamma_1}(1-s)^{-\gamma_2},1)\dif s
        \\&
        \lesssim \int_0^{1/2}s^{-\alpha_1}\min(\varepsilon s^{-\gamma_1},1)\dif s+\int_{1/2}^1(1-s)^{-\alpha_2}\min(\varepsilon (1-s))^{-\gamma_2},1)\dif s
        \\&
        = \int_0^{1/2}s^{-\alpha_1}\min(\varepsilon s^{-\gamma_1},1)\dif s+\int_0^{1/2}s^{-\alpha_2}\min(\varepsilon s^{-\gamma_2},1)\dif s.
    \end{align*}
    It is straightforward to verify that 
    \begin{align*}
        \int_0^{1/2}s^{-\alpha}\min(\varepsilon s^{-\gamma},1)\dif s\lesssim \ell_{\alpha,\gamma}(\eps)
    \end{align*}
    for all $\eps>0$, $\alpha\in[0,1)$ and $\gamma\ge0$.
    This completes the proof.
\end{proof}

\bl
Let $\alpha>\beta>0$. Then there is a constant $C=C(\alpha,\beta)>0$ such that for any $\eps>0$,
\begin{align}\label{sum_inter}
    \sum_{j=-1}2^{-\beta j}\left(1\wedge (2^{\alpha j}\eps)\right)\le C \eps^{\frac{\beta}{\alpha}}.
\end{align}
\el
\begin{proof}
It is easy to see that
 \begin{align*}
\sum_{j=-1}^\infty 2^{-\beta j}\left(1\wedge \left(\eps 2^{\alpha j}\right) \right)&\le\sum_{j=-1}^\infty 2^{\beta+1}\int_{2^{j}}^{2^{j+1}}s^{-\beta-1}(1\wedge (\eps s^\alpha))ds\lesssim \int_0^\infty s^{-\beta-1}\left(1\wedge\left(\eps s^\alpha\right)\right)\dif s\\
        &\lesssim\eps^{\frac\beta\alpha} \int_0^\infty s^{-\beta-1}(1\wedge s^\alpha)\dif s\lesssim \eps^{\frac\beta\alpha}.
    \end{align*}    
\end{proof}

\bl
Let $s\in(0,1)$ and $\bbp\in[1,\infty]^2$. There is a constant $C=C(d,s,\bbp)>0$ such that for all $h\in\mR^d$ and $f\in \bB^{3s}_{\bbp;\bba}$,
\begin{align}\label{xy-difference}
    \|f(\cdot+h,\cdot)-f(\cdot,\cdot)\|_{\mL^{\bbp}}\le C|h|^s\|f\|_{\bB^{3s}_{\bbp;\bba}}.
\end{align}
\el
\begin{proof}
    We note that for any $x,v\in\mR^d$,
    \begin{multline*}
        |f(x+h,v)-f(x,v)|\le \sum_{j=-1}^\infty|\cR_j^\bba f(x+h,v)-\cR_j^\bba f(x,v)|\\
        \le\sum_{j=-1}^\infty\left(|\cR_j^\bba f(x+h,v)|+|\cR_j^\bba f(x,v)|\right)\wedge \left(h\cdot\int_0^1\nabla_x\cR_j^\bba f(x+sh,v)d s \right).
       \end{multline*}
    Then it follows from Bernstein's inequalities (see \eqref{Ber}) that    
       \begin{align*}   
        \|f(\cdot+h,\cdot)-f(\cdot,\cdot)\|_{\mL^\bbp}&\lesssim \sum_{j=-1}^\infty\left(|h|\|\nabla_x\cR_j^\bba f\|_{\mL^\bbp}\right)\wedge \|\cR_j^\bba f\|_{\mL^\bbp}\\
        &\lesssim \sum_{j=-1}^\infty\left(\left(|h|2^{3j}\right)\wedge 1 \right)\|\cR_j^\bba f\|_{\mL^\bbp}\lesssim \sum_{j=-1}^\infty\left(\left(|h|2^{3j}\right)\wedge 1 \right)2^{-3sj} \|f\|_{\bB^{3s}_{\bbp;\bba}}\\
        &\lesssim|h|^{s}\|f\|_{\bB^{3s}_{\bbp;\bba}},
    \end{align*}
    provided by \eqref{sum_inter}, and the proof completes.
\end{proof}

    \begin{lemma}
        \label{lem:f-fn}
        Let $f$ be a measurable function defined on $\mR^{2d}$, define $f_n:=f\ast\varphi_n$ where $\varphi_n(x,v):=n^{4d\vartheta}\varphi(n^{3\vartheta}x, n^{\vartheta}v),$ $\vartheta\geq0$ and $\varphi$ is a probability density function on $\mR^{2d}$ with $\varphi(x,v)=\varphi(x,-v)$.  Then for any $s\in\mR$, $\beta\in(0,2)$, $\bbp\in[1,\infty)^2$,
         we have
         \begin{align}\label{est:f-fn-app}
             \|f-f_n\|_{\bB^{s,1}_{\bbp;\bba}}\lesssim n^{-\vartheta \beta} \|f\|_{\bB^{s+\beta}_{\bbp;\bba}}.
         \end{align}
    \end{lemma}
\begin{proof}
    By \cref{bs}, $ \|f-f_n\|_{\bB^{s}_{\bbp;\bba}}=\sup_j2^{sj}   \|\cR_j^\bba(f-f_n)\|_{\bbp}$.
    Since $\varphi_n$ is a probability density function, 
    for any $(x,v)\in\mR^{2d}$, we have
    \begin{align*}
        \cR_j^\bba(f-f_n)(x,v)=&\int_{\mR^{2d}}\varphi_n(y,w)\big( \cR_j^\bba f(x,v)-\cR_j^\bba f(x-y,v-w)\big)\dif y\dif w\\
        =&\int_{\mR^{2d}}\varphi_n(y,w)\big( \cR_j^\bba f(x,v)-\cR_j^\bba f(x,v-w)\big)\dif y\dif w\\
        &+\int_{\mR^{2d}}\varphi_n(y,w)\big( \cR_j^\bba f(x,v-w)-\cR_j^\bba f(x-y,v-w)\big)\dif y\dif w\\
        =:&I_1(j)+I_2(j).
    \end{align*}
Using the symmetry $\varphi_n(x,v)=\varphi_n(x,-v)$ and Bernstein's inequalities (see \eqref{Ber}) one sees that 
\begin{align*}
        \| I_1(j)\|_\bbp=&\left\|\int_{\mR^{2d}}\varphi_n(y,w)\big(\cR_j^\bba f(x,v+w)+\cR_j^\bba f(x,v-w)-2\cR_j^af(x,v)\big)\dif y\dif w\right\|_\bbp\\
        \lesssim& \int_{\mR^{2d}}|\varphi_n(y,w)||w|^2\| \nabla^2\cR_j^\bba f\|_{\mL^\bbp}\dif y\dif w
        \lesssim 2^{2j}\| \cR_j^\bba f\|_{\mL^\bbp} n^{-2\vartheta}.
    \end{align*}
     In view of \eqref{xy-difference}, we have
 \begin{align*}
     \| I_2(j)\|_\bbp&\lesssim \int_{\mR^{2d}}|\varphi_n(y,w)||y|^\frac{2}{3}\| \cR_j^\bba f\|_{{\bB^{2}_{\bbp;\bba}}}\dif y\dif w\lesssim\| \cR_j^\bba f\|_{\bB^{2}_{\bbp;\bba}} n^{-2\vartheta}\\
     &\lesssim  \sup_k2^{2k}   \|\cR_k^a\cR_j^\bba f\|_{\bbp} n^{-2\vartheta}\lesssim \sup_{k\sim j}2^{2k}   \|\cR_j^\bba f\|_{\bbp} n^{-2\vartheta}\lesssim 2^{2j}   \|\cR_j^\bba f\|_{\bbp} n^{-2\vartheta}.
 \end{align*}
 Moreover, it is easy to see that
 \begin{align*}
     \|\cR_j^\bba(f-f_n)\|_{\bbp}\le \|\cR_j^\bba f\|_{\bbp}+\|\cR_j^\bba f_n\|_{\bbp}\lesssim \|\cR_j^\bba f\|_{\bbp}.
 \end{align*}
    Therefore, we get 
     \begin{align*}
      \|f-f_n\|_{\bB^{s,1}_{\bbp;\bba}}&=\sum_{j=-1}^\infty2^{sj}   \|\cR_j^\bba(f-f_n)\|_{\bbp}\lesssim \sum_{j=-1}^\infty 2^{sj} \|\cR_j^\bba f\|_{\bbp}\left(\left(2^{2j}n^{-2\vartheta}\right)\wedge 1\right)\\
      &\lesssim \sum_{j=-1}^\infty 2^{-\beta j}\left(\left(2^{2j}n^{-2\vartheta}\right)\wedge 1\right)\|f\|_{\bB^{s+\beta}_{\bbp;\bba}},
    \end{align*}
    which by \eqref{sum_inter} implies \eqref{est:f-fn-app}.
\end{proof}
    \bl\label{App:cutoff}
Let $N>1$ and $\bbp=(p_x,p_v)\in (1,\infty]^2$ such that $\bbp\neq (\infty,\infty)$. 
For any $\delta\in(0,((p_x\wedge p_v)-1)(\bba\cdot\frac{d}{\bbp}))$, 
there is a constant $C=C(d,\bbp,\delta)>0$ such that for any $f,g\in \mL^{\bbp}$ satisfying
$
|g|\le |f|1_{|f|>N}
$,
one has
\begin{align}\label{App:cutoff00}
\|g\|_{\bB^{-\delta,1}_{\bbp;\bba}}\le CN^{-\delta(\bba\cdot d/\bbp)^{-1}}\|f\|_{\mL^{\bbp}}^{1+\delta(\bba\cdot d/\bbp)^{-1}}.
\end{align}
\el
\begin{proof}

We consider three cases.


{\bf Case (i)} ($p_x=\infty$, $p_v\ne\infty$). 
Let $q\in[1,p_v)$ be such that $\delta+\frac{d}{p_v}=\frac{d}{q}$. Using embedding $\bB^{0,1}_{(\infty,q);\bba}\hookrightarrow\bB^{-\delta,1}_{\bbp;\bba}$ (\cite[Appendix B]{HRZ}) and \eqref{AB2}, we have
\begin{align*}
  \|g\|_{\bB^{-\delta,1}_{\bbp;\bba}}\lesssim \|g\|_{\bB^{0,1}_{(\infty,q);\bba}}\lesssim\|g\|_{\mL^{(\infty,q)}} \le\|f1_{|f|>N}\|_{\mL^{(\infty,q)}}.  
\end{align*}
Note that
\begin{align*}
    \|f(\cdot,v)1_{|f(\cdot,v)|>N}\|_{L^\infty}\le \|f(\cdot,v)\|_{L^\infty}1_{\{\|f(\cdot,v)\|_{L^\infty}>N\}}.
\end{align*}
Setting $h(v):=\|f(\cdot,v)\|_{L^\infty}$, we have
\begin{align*}
  \|f1_{|f|>N}\|_{\bB^{-\delta,1}_{\bbp;\bba}}&\lesssim  \|h1_{|h|>N}\|_{L^q}
  \le N^{-\frac{p_v-q}{q}}\|h\|_{L^{p_v}}^{\frac{p_v}{q}}
  =N^{-\frac{\delta p_v}{d}}\|f\|_{\mL^{\bbp}}^{\frac{p_v}{q}}.
\end{align*}


{\bf Case (ii)} ($p_v=\infty$). 
Let $q\in[1,p_x)$ be such that $\delta+\frac{3d}{p_x}=\frac{3d}{q}$.  Using embedding $\bB^{0,1}_{(q,\infty);\bba}\hookrightarrow\bB^{-\delta,1}_{\bbp;\bba}$ (\cite[Appendix B]{HRZ}) and \eqref{AB2}, we have
\begin{align*}
  \|g\|_{\bB^{-\delta,1}_{\bbp;\bba}}\lesssim \|g\|_{\bB^{0,1}_{(q,\infty);\bba}}\lesssim\|f1_{|f|>N}\|_{\mL^{(q,\infty)}}.  
\end{align*}
For each $v\in\mR^d$, we have
\begin{align*}
    \|f(\cdot,v)1_{|f(\cdot,v)|>N}\|_{L^{q}}
    \le N^{-\frac{\delta p_x}{3d}}\|f(\cdot,v)\|_{L^{p_x}}^{\frac{p_x}{q}}.
\end{align*}
This gives \eqref{App:cutoff00} upon taking supremum over $v$.

\vspace{1mm}

{\bf Case (iii)} ($p_x,p_v\ne\infty$). We put $ k=1+\delta(\bba\cdot\tfrac{d}{\bbp})^{-1}$. 
Since $\delta\in(0,((p_x\wedge p_v)-1)(\bba\cdot\frac{d}{\bbp}))$, we see that $\bbq:=(q_x,q_v)=(\frac{p_x}{k},\frac{p_v}{k})\in[1,\infty)^2$ and
\begin{align*}
    \delta+\tfrac{3d}{p_x}+\tfrac{d}{p_v}=\tfrac{3d}{q_x}+\tfrac{d}{q_v} .
\end{align*}
Then by embedding $\bB^{0,1}_{\bbq;\bba}\hookrightarrow\bB^{-\delta,1}_{\bbp;\bba}$ (\cite[Appendix B]{HRZ}) and \eqref{AB2}, we have
\begin{align*}
  \|g\|_{\bB^{-\delta,1}_{\bbp;\bba}}\lesssim \|g\|_{\bB^{0,1}_{\bbq;\bba}}\lesssim\|g\|_{\mL^\bbq}\lesssim\|f1_{|f|>N}\|_{\mL^\bbq}.  
\end{align*}
It is straightforward to see that
\begin{align*}
    \|f1_{|f|>N}\|_{\mL^\bbq}
    \le N^{-\frac{p_x-q_x}{q_x}}\left(\int_{\mR^d}\|f(\cdot,v)\|_{\mL^{p_x}}^{\frac{p_x q_v}{q_x}}\dif v\right)^{\frac{1}{q_v}}=
    N^{-\frac{p_x-q_x}{q_x}}\|f\|_{\mL^\bbp}^{\frac{p_v}{q_v}}.
\end{align*} 
Noting that $\tfrac{p_x-q_x}{q_x}=k-1=\delta(\bba\cdot\tfrac{d}{\bbp})^{-1}$, 
the proof completes.
\end{proof}

\section{Paraproduct estimates  in anisotropic Besov spaces}\label{sec:3.2}

Recall the Bony decomposition
\begin{align}\label{def:Bony}
    fg=f\prec g+f\circ g+f\succ g=:f\prec g+f\succcurlyeq g,
\end{align}
where
\begin{align*}
   f\prec g=g\succ f:=\sum_{k=1}^\infty S_{k-1}^\bba f\cR_k^\bba g,
   \quad S_k^\bba f:=\sum_{j=0}^k\cR_j^\bba f, \quad 
   f\circ g:=\sum_{|i-j|\le2} \cR_i^\bba f\cR_j^\bba g. 
\end{align*}
\begin{lemma}[Paraproduct estimates]\label{lem:A2}
   Let  $\bbp,\bbq,\bbr\in[1,\infty]^2$ with $1/\bbp+1/\bbq=1/\bbr, \alpha\in\mR$ and  $\beta<0.$ Then
\begin{align}\label{0426:09}
    \|f\prec g\|_{\bB^{\alpha}_{\bbr;\bba}}\lesssim \|f\|_{\bbq}\|g\|_{\bB^{\alpha}_{\bbp;\bba}}
\end{align}
and
\begin{align}\label{0426:08}
    \|f\succ g\|_{\bB^{\alpha+\beta}_{\bbr;\bba}}\lesssim \|f\|_{\bB^{\alpha}_{\bbp;\bba}}\|g\|_{\bB^{\beta}_{\bbq;\bba}}.
\end{align}
When $ \alpha+\beta>0,$ 
\begin{align}\label{0426:10}
    \|f\circ g\|_{\bB^{\alpha+\beta}_{\bbr;\bba}}\lesssim \|f\|_{\bB^{\alpha}_{\bbp;\bba}}\|g\|_{\bB^{\beta}_{\bbq;\bba}}.
\end{align}
Moreover, we have
\begin{gather}
    \label{0426:07}
    \|f\circ g\|_{\mL^\bbr}\lesssim \|f\|_{\bB^{\alpha,1}_{\bbp;\bba}}\|g\|_{\bB^{-\alpha}_{\bbq;\bba}},
   \\ \label{0426:07succ}
    \|f\succ g\|_{\mL^\bbr}\le \|f\|_{\bB^{-\beta,1}_{\bbp;\bba}}\|g\|_{\bB^{\beta,1}_{\bbq;\bba}},
  \\  \label{0426:07succ0}
    \|f\succ g\|_{\mL^\bbr}\le \|f\|_{\bB^{0,1}_{\bbp;\bba}}\|g\|_{\bbq}.
\end{gather}
Consequently, 
\begin{align}\label{est:Paraconse}
   \|f\succcurlyeq g\|_{\bbr}\lesssim \min(\|f\|_{\bB^{0,1}_{\bbq;\bba}}\|g\|_{\bbp},\|f\|_{\bB^{-\beta,1}_{\bbq;\bba}}\|g\|_{\bB^{\beta,1}_{\bbp;\bba}}).
\end{align}
\end{lemma}
\begin{proof}
    The estimates \eqref{0426:09}, \eqref{0426:08} and \eqref{0426:10} are standard (see \cite[Lemma 2.11]{HZZZ22} for instance). We only show \eqref{0426:07}-\eqref{0426:07succ0}, which implies \eqref{est:Paraconse} through the embedding $\LL^\bbq\hookrightarrow \bB^{0}_{\bbq;\bba}$ (see \eqref{AB2}). 
    Concerning \eqref{0426:07}, by definition and H\"older inequality, we have that
    \begin{align*}
        \|f\circ g\|_{\bbr}\lesssim \sum_{i=0}^\infty\sum_{\ell=-2}^2\|\cR_{i}^\bba f\|_{\bbp}\|\cR_{i+\ell}^\bba g\|_{\bbq}\lesssim \sum_{i=0}^\infty2^{\alpha i}\|\cR_{i}^\bba f\|_{\bbp}\|g\|_{\bB^{-\alpha}_{\bbq;\bba}}\lesssim \|f\|_{\bB^{\alpha,1}_{\bbp;\bba}}\|g\|_{\bB^{-\alpha}_{\bbq;\bba}}.
    \end{align*}
    As for \eqref{0426:07succ0} and \eqref{0426:07succ}, applying H\"older inequality, we have
\begin{align*}
    \|f\succ g\|_{\mL^\bbr}\le\sum_{k=1}^\infty\|S^\bba_{k-1}g \cR^\bba_k f\|_\bbr\le\sum_{k=1}^\infty\|S^\bba_{k-1}g\|_{\bbq} \|\cR^\bba_k f\|_\bbp.
\end{align*}
 From \eqref{AA13}, we have $ \widehat{S^\bba_kg}(\xi)= \chi^\bba_0(2^{-k\bba}\xi)\hat{g}(\xi)$.
This implies, through  Young convolution inequality, that  $\|S^\bba_kg\|_{\bbp}  \le \|g\|_\bbp$.
Hence, we have 
\begin{align*}
    \|f\succ g\|_{\mL^\bbr}
    \le\sum_{k=1}^\infty\|g\|_{\bbq} \|\cR^\bba_k f\|_\bbp
    \le \|f\|_{\bB^{0,1}_{\bbp;\bba}}\|g\|_{\bbq},
\end{align*}
which shows \eqref{0426:07succ0}. Furthermore, since $\beta<0$, we have $\|S^\bba_kg\|_{\bbp}  \le 2^{-\beta k} \|g\|_{\bB^{\beta,1}_{\bbq;\bba}}$.
This yields that
    \begin{align*}
    \|f\succ g\|_{\mL^\bbr}
    \le \sum_{k=1}^\infty2^{-\beta k} \|\cR^\bba_k f\|_\bbp\|g\|_{\bB^{\beta,1}_{\bbq;\bba}} \le \|f\|_{\bB^{-\beta,1}_{\bbp;\bba}}\|g\|_{\bB^{\beta,1}_{\bbq;\bba}},
\end{align*}
which shows \eqref{0426:07succ}.
    \end{proof}
   
\section*{Acknowledgment}
KL has been funded by the Engineering \& Physical Sciences Research Council (EPSRC) Grant EP/Y016955/1.  CL is supported by Deutsche Forschungsgemeinschaft (DFG) - Projektnummer 563883019.
We are thankful to Prof. Fengyu Wang (Tianjin University, China) who suggested this question to us.
\bibliographystyle{alpha}
\bibliography{references}

@article {OJ,
    AUTHOR = {Bencheikh, O. and Jourdain, B.},
     TITLE = {Convergence in total variation of the {E}uler-{M}aruyama
              scheme applied to diffusion processes with measurable drift
              coefficient and additive noise},
   JOURNAL = {SIAM J. Numer. Anal.},
  FJOURNAL = {SIAM Journal on Numerical Analysis},
    VOLUME = {60},
      YEAR = {2022},
    NUMBER = {4},
     PAGES = {1701--1740},
      ISSN = {0036-1429,1095-7170},
   MRCLASS = {60H35 (60H10 65C05 65C30)},
  MRNUMBER = {4451313},
MRREVIEWER = {Sa\'ul\ D\'iaz-Infante},
       DOI = {10.1137/20M1371774},
       URL = {https://doi.org/10.1137/20M1371774},
}

@article{talay1990expansion,
  title={Expansion of the global error for numerical schemes solving stochastic differential equations},
  author={Talay, D. and Tubaro, L.},
  journal={Stochastic analysis and applications},
  volume={8},
  number={4},
  pages={483--509},
  year={1990},
  publisher={Taylor \& Francis}
}

@article{B,
  author = {F. Bouchut},
  title = {{Hypoelliptic regularity in kinetic equations}},
  journal = {J. Math Pures Appl.},
  volume = {81},
  pages = {1135--1159},
  year = {2002},
  doi = {10.1016/S0021-7824(02)01255-1}
}

@article{R,
  author = {P.-E. Chaudru de Raynal},
  title = {{Strong existence and uniqueness for degenerate SDE with H\"oder drift}},
  journal = {Ann. Inst. Henri Poincaré Probab. Stat.},
  volume = {53},
  pages = {259--286},
  year = {2017},
  doi = {10.1214/16-AIHP721}
}

@article{RM,
  author = {P. Chaudru De Raynal and S. Menozzi},
  title = {{Regularization effects of a noise propagating through a chain of differential equations: an almost sharp result}},
  journal = {Trans. Amer. Math. Soc.},
  volume = {375},
  pages = {1-45},
  year = {2022},
  doi = {10.1090/tran/8576}
}

@article{HZZZ22,
  author = {Z. Hao and X. Zhang and R. Zhu and X. Zhu},
  title = {{Singular kinetic equations}},
  journal = {Ann. Probab.},
  volume = {52},
  pages = {576--657},
  year = {2024},
  doi = {10.1214/22-AOP1551}
}

@article {HRZ,
    AUTHOR = {Hao, Z. and R\"ockner, M. and Zhang, X.},
     TITLE = {Second-order fractional mean-field {SDE}s with singular
              kernels and measure initial data},
   JOURNAL = {Ann. Probab.},
  FJOURNAL = {The Annals of Probability},
    VOLUME = {54},
      YEAR = {2026},
    NUMBER = {1},
     PAGES = {1--62},
      ISSN = {0091-1798,2168-894X},
   MRCLASS = {60H10 (35R11 35R60)},
  MRNUMBER = {5019007},
       DOI = {10.1214/24-AOP1709},
       URL = {https://doi.org/10.1214/24-AOP1709},
}

@article{HJK,
  author = {M. Hutzenthaler and A. Jentzen and P. Kloeden},
  title = {{Strong convergence of an explicit numerical method for SDEs with nonglobally Lipschitz continuous coefficients}},
  journal = {Ann. Appl. Probab.},
  volume = {22},
  pages = {1611--1641},
  year = {2012},
  doi = {10.1214/11-AAP803}
}

@article{Holland,
  author = {T. Holland},
  title = {{A note on the weak rate of convergence for the Euler-Maruyama scheme with H\"older drift}},
  journal = {Stoch. Proc. Appl.},
  volume = {174},
  year = {2024},
  doi = {10.1016/j.spa.2023.10.001}
}

@article{H,
  author = {L. H\"ormander},
  title = {{Hypoelliptic second order differential equations}},
  journal = {Acta Math.},
  volume = {119},
  pages = {147--171},
  year = {1967},
  doi = {10.1007/BF02392081}
}

@article{JM21,
  author = {B. Jourdain and S. Menozzi},
  title = {{Convergence Rate of the Euler-Maruyama Scheme Applied to Diffusion Processes with $L^q-L^\rho$ Drift Coeffcient and Additive Noise}},
  journal = {Ann. Appl. Probab.},
  volume = {34},
  pages = {1663--1697},
  year = {2024},
  doi = {10.1214/23-AAP1876}
}

@article{LM,
  author = {V. Lemaire and S. Menozzi},
  title = {{On some Non Asymptotic Bounds for the Euler Scheme}},
  journal = {Electron. J. Probab.},
  volume = {15},
  pages = {1645--1681},
  year = {2010},
  doi = {10.1214/EJP.v15-822}
}

@article{Le2022,
  author = {K. L{\^e}},
  title = {{Quantitative John--Nirenberg inequality for stochastic processes of bounded mean oscillation}},
  journal = {arXiv preprint},
  year = {2022},
  eprint = {https://arxiv.org/pdf/2210.15736}
}

@article {LL,
    AUTHOR = {L\^e, K. and Ling, C.},
     TITLE = {Taming singular stochastic differential equations: a numerical
              method},
   JOURNAL = {Ann. Probab.},
  FJOURNAL = {The Annals of Probability},
    VOLUME = {53},
      YEAR = {2025},
    NUMBER = {5},
     PAGES = {1764--1824},
      ISSN = {0091-1798,2168-894X},
   MRCLASS = {60H35 (35R60 60H10 60H50 60L90)},
  MRNUMBER = {4962731},
       DOI = {10.1214/24-aop1750},
       URL = {https://doi.org/10.1214/24-aop1750},
}

@article {RZ24,
    AUTHOR = {Ren, C. and Zhang, X.},
     TITLE = {Heat kernel estimates for kinetic {SDE}s with drifts being
              unbounded and in {K}ato's class},
   JOURNAL = {Bernoulli},
  FJOURNAL = {Bernoulli. Official Journal of the Bernoulli Society for
              Mathematical Statistics and Probability},
    VOLUME = {31},
      YEAR = {2025},
    NUMBER = {2},
     PAGES = {1402--1427},
      ISSN = {1350-7265,1573-9759},
   MRCLASS = {60H10},
  MRNUMBER = {4863081},
MRREVIEWER = {Feng-Yu\ Wang},
       DOI = {10.3150/24-bej1775},
       URL = {https://doi.org/10.3150/24-bej1775},
}

@article{RE,
  author = {L. Rothschild and E. Stein},
  title = {{Hypoelliptic differential operators and nilpotent groups}},
  journal = {Acta Math.},
  volume = {137},
  pages = {247--320},
  year = {1976},
  doi = {10.1007/BF02392360}
}

@book {S,
    AUTHOR = {Soize, C.},
     TITLE = {The {F}okker-{P}lanck equation for stochastic dynamical
              systems and its explicit steady state solutions},
    SERIES = {Series on Advances in Mathematics for Applied Sciences},
    VOLUME = {17},
 PUBLISHER = {World Scientific Publishing Co., Inc., River Edge, NJ},
      YEAR = {1994},
     PAGES = {xvi+321},
      ISBN = {981-02-1755-2},
   MRCLASS = {60J60 (34F05 35R60 70L05 73K35 82C31 93E03)},
  MRNUMBER = {1287386},
MRREVIEWER = {Wolfgang Kliemann},
       DOI = {10.1142/9789814354110},
       URL = {https://doi.org/10.1142/9789814354110},
}

@book{Tri06,
  author = {H. Triebel},
  title = {{Theory of function spaces, III}},
  publisher = {Birkh\"{a}user},
  address = {Basel},
  year = {2006},
  doi = {10.1007/978-3-0348-5255-8}
}

@article{T,
  author = {D. Talay},
  title = {{Stochastic Hamiltonian Systems: Exponential Convergence to the Invariant Measure, and Discretization by the Implicit Euler Scheme}},
  journal = {Markov Processes Relat. Fields},
  volume = {8},
  pages = {1--36},
  year = {2002}
}

@article{V,
  author = {C. Villani},
  title = {{A review of mathematical topics in collisional kinetic theory}},
  journal = {{Handbook of mathematical fluid dynamics}},
  year = {2002}
}

@article{Zhang10,
  author = {X. Zhang},
  title = {{Stochastic Volterra equations in Banach spaces and stochastic partial differential equations}},
  journal = {J. Fun. Anal.},
  volume = {258},
  pages = {1361--1425},
  year = {2010},
  doi = {10.1016/j.jfa.2009.10.014}
}

@article{ZZ21,
  author = {X. Zhang and X. Zhang},
  title = {{Cauchy problem of stochastic kinetic equations}},
  journal = {Ann. Appl. Probab.},
  volume = {34},
  pages = {148--202},
  year = {2024},
  doi = {10.1214/23-AAP1876}
}

\end{document}